\documentclass[]{article}
\usepackage{graphicx}
\usepackage{caption}
\usepackage{amsfonts}
\usepackage{amsmath}
\usepackage{amssymb}
\usepackage{fancyhdr}
\usepackage{titlesec}
\usepackage{indentfirst}
\usepackage{booktabs}
\usepackage{verbatim}
\usepackage{color}
\usepackage{tcolorbox}
\usepackage{mathtools}
\usepackage{bm}
\usepackage{amsthm}
\usepackage{wasysym}
\usepackage{empheq}
\graphicspath{{figures/}} 
\usepackage[font=small,labelfont=bf]{caption} 
\usepackage{chngcntr}

\newcommand{\rom}[1]{\uppercase\expandafter{\romannumeral #1\relax}}
\usepackage[page,header]{appendix}
\usepackage{subfig}

\usepackage{titletoc}

\numberwithin{equation}{section}
\newtheorem{theorem}{Theorem}[section]
\newtheorem{lemma}[theorem]{Lemma}

\newtheorem{remark}[theorem]{Remark}
\newtheorem{condition}[theorem]{Assumption}
\newtheorem{algorithm}[theorem]{Algorithm}
\newtheorem{gallarynotation}[theorem]{Gallery of Notations}

\newcommand{\argmin}{\operatornamewithlimits{argmin}}
\newcommand{\argmax}{\operatornamewithlimits{argmax}}

\def\thetheorem {{\arabic{section}.\arabic{theorem}}}

\usepackage{tikz}
\newcommand*\circled[1]{\tikz[baseline=(char.base)]{
            \node[shape=circle,draw,inner sep=0.8pt] (char) {#1};}}

\def\a{\alpha}
\def\b{\beta}
\def\d{\delta}
\def\e{\epsilon}
\def\g{\gamma}
\def\lam{\lambda}
\def\o{\omega}
\def\p{\phi}
\def\s{\sigma}
\def\th{\theta}

\def\th{\theta}
\def\L{\Lambda}
\def\G{\Gamma}
\def\O{\Omega}

\def\bz{{\bf{z}}}

\def\bg{\bm{\gamma}}

\def\hf{\hat{f}}

\def\R{\mathbb R}
\def\N{\mathbb N}
\def\E{\mathbb E}

\def\P{\mathbb{P}}
\def\mL{\mathcal{L}}
\def\mP{\mathcal{P}}
\def\mO{\mathcal{O}}

\def\mF{\mathcal{F}}
\def\mQ{\mathcal{Q}}
\def\mS{\mathcal{S}}

\def\l{\left}
\def\r{\right}
\def\la{\left\langle}
\def\ra{\right\rangle}
\def\ll{\left\lVert}
\def\rl{\right\rVert}
\def\lv{\left\lvert}
\def\rv{\right\rvert}
\def\({\left(}
\def\){\right)}
\def\[{\left[}
\def\]{\right]}

\def\pt{\partial}

\def\qd{\quad}

\def\t{\tilde}

\def\sM{\sqrt{M}}
\def\h{\hat}

\def\bE{\bar{E}}

\def\pinf{\p^\infty}
\def\Einf{E^\infty}
\def\einf{E^\infty}
\def\thi{\th_i}

\begin{document}

\title{The Vlasov-Fokker-Planck Equation with High Dimensional Parametric Forcing Term}

\author{ 
Shi Jin\footnote{School of Mathematical Sciences, Institute of Natural Sciences, MOE-LSEC and SHL-MAC, Shanghai Jiao Tong University, Shanghai 200240, China. (shijin-m@sjtu.edu.cn) }\ \ \ \ \  Yuhua Zhu\footnote{Department of Mathematics, University of Wisconsin-Madison, Madison, WI 53706, USA. (yzhu232@wisc.edu)}  \ \  \  \ \   
Enrique Zuazua\footnote{Chair in Applied Analysis, Alexander von Humboldt-Professorship, Department of Mathematics Friedrich-
Alexander-Universit\"at Erlangen-N\"urnberg, 91058 Erlangen, Germany.}  
\footnote{Chair of Computational Mathematics, Fundaci\'on Deusto, Av. de las Universidades 24, 48007 Bilbao, Basque Country, Spain.} 
\footnote{Departamento de Matem\'aticas, Universidad Aut\'onoma de Madrid, 28049 Madrid, Spain. }
}

\date{}
\maketitle

\begin{abstract}
  We consider the Vlasov-Fokker-Planck equation with random electric field
  where the random field is parametrized by  countably many infinite random variables due to uncertainty.  
  At the theoretical level, with suitable assumption on the anisotropy of the randomness, adopting the technique employed in elliptic PDEs \cite{cohen2015approximation}, we prove the best N approximation in the random space breaks the dimension curse and the convergence rate is faster than the Monte Carlo method. For the numerical method, based on the adaptive sparse polynomial interpolation (ASPI) method introduced in \cite{chkifa2014high}, we develop a residual based adaptive sparse polynomial interpolation (RASPI) method which is more efficient for multi-scale linear kinetic equation, when using numerical schemes that
  are time dependent and implicit. Numerical experiments
  show that  the numerical error of the RASPI decays faster than the
  Monte-Carlo method and is also dimension independent.
\end{abstract}

{\small


}

\section{Introduction}
\label{sec:Intro}
We consider the Vlasov-Fokker-Planck (VFP) equation with a random electric field due
to uncertainty.  Typically uncertainty is modeled by a stochastic field,
which by the Karhunen-L\`oeve approximation is parametrized by  countably  infinite random variables \cite{loeve1978probability1, loeve1978probability2}. One of the difficulties in the development of  numerical methods for such
problems is the possible curse of dimension.  Sampling methods, such as
the Monte-Carlo methods, are often used,  which are  dimension
independent. However, these methods suffer from the low convergence rate,
as the numerical errors are of  $O(N^{-\frac{1}{2}})$ with $N$ sample points. In this paper we seek a more efficient numerical method based on best $N$ approximation and greedy algorithms, originally developed for elliptic equations,
for uncertain VFP equation in which the electric field depends on high
dimensional random variables in order to achieve a numerical convergence rate
faster than the Monte-Carlo methods.

There are two separated parts in this paper. In Section \ref{theoretical}, we reviewe  the best N approximation, and  prove  the convergence rate of it when
applied to the VPF equation, under suitable assumption on random field.
We point out in Section \ref{numerical}, the best N approximation is a non-linear approximation  hard to implement in practice. Therefore, for the numerical method, we develop a residual based adaptive sparse polynomial interpolation (RASPI) method, which is shown in Section \ref{numerical eq} by different numerical experiments to be efficient in practice and indeed converges faster than the Monte Carlo method.

Our theoretical results in Section \ref{theoretical} are based on the results in a series of paper \cite{chkifa2015breaking, cohen2010convergence, cohen2011analytic, cohen2015approximation}. For uniformly distributed random variables,
we seek approximate solutions in a finite dimensional space spanned by the  Legendre polynomial basis, that is $u(\bz)\approx u_\L(\bz) = \sum_{\nu\in\L}\o_\nu L_\nu(\bz)$ for $\#(\L) = N$, where $\#(\L)$ is the number of elements in $\Lambda$. The best N approximation is to truncate the basis according to the $N$ largest coefficient $\o_\nu$, so that the mean square error, which is represented by $\sum_{\nu\notin\L}| \o_\nu |^2$, can be as small as possible. It is summarized in \cite{cohen2015approximation} that holomorphy and anisotropy of the solution in the random space implies that the best N approximation can  break the curse of  dimension. Furthermore it converges faster than the usual Monte Carlo Method. It has been successfully applied to elliptic equations, including parametric PDEs,  control problems,  inverse problems, etc \cite{hansen2013analytic, hansen2013sparse, kunoth2013analytic, schillings2013sparse, schillings2014sparsity, schwab2012sparse, DJL19, liu2018hypocoercivity}. However, since this result requires analyticity of the solution in the random space, it hasn't been widely used in other PDEs, such as  kinetic and related equations. Thanks to recent studies on the regularity of the solution to most of the kinetic equations using hypocoercivity
of the kinetic operators, which give rise to high order  regularity in the random space
for kinetic equations with uncertainties, if the (random)  initial data and
(random) coefficients
have such regularities,  \cite{jin2016august,li2016uniform, ZhuJin18, Jintwophase, Zhu2017BE, Jin-LiuL}, we first extend such reegularity study for the
VFP equation  with {\it infinite} dimensional random variables, which
gives  the {\it first} error estimate for uncertain kinetic equation that is
{\it independent} of the dimension of the random variables. While for moderately high dimensionality, sparse grids were used for uncertain kinetic equations \cite{hu2019stochastic}, here we are interested in much high dimensions in the the random space.

Based on the theoretical results, we develop a numerical method in Section \ref{numerical}, which is then applied in Section \ref{numerical eq}  to several examples to verify that it indeed successfully breaks the curse of dimension. The numerical method we develop is a residual based adaptive sparse polynomial interpolation (RASPI) method, which combines the idea from the adaptive sparse polynomial interpolation (ASPI) method and the residual based greedy search. The ASPI method, introduced in \cite{chkifa2014high} in  line of \cite{nobile2008sparse,nobile2008anisotropic}, is a non-intrusive method that computes a polynomial approximation by interpolation of the solution map at $N$ well chosen points. In particular it could be applied in the case when the exact model is not known,
and only the numerical solver is given. However, in order to find the  ``well chosen" points, one needs to calculate the solution at number of sample points much bigger than $N$, which can be very costly when the PDE is time dependent and the numerical scheme is in implicit form. Actually for most multi-scale kinetic equations, one indeed needs to use  implicit schemes due to the presence of small parameter or numerical stiffness \cite{jin2010asymptotic, jin2011asymptotic,filbet2010class,jin2011class}. This means one needs to invert the approximate kinetic operator, for each
mesh point, which is a large matrix in each time step. As will be shown later, the inversion of a large matrix can be avoided by computing the residual of the PDE instead. The idea of using residual of a PDE has been used in greedy algorithm \cite{rozza2007reduced} for parametric PDEs and recently also applied to parametric control problems \cite{lazar2016greedy, hernandez2017greedy}, but
mainly  in low dimensions. Although such a method may end up using even less basis compared to polynomial approximations, their offline stage is potentially very costly, especially in high dimensions. The RASPI method combines the advantages of both methods, so that one can save much computational cost by calculating the residual of the PDE instead of the numerical solution of the PDE. At the same time, the offline stage is still efficient in high dimensions.

We would like  to point out that although all numerical experiments in Section \ref{numerical eq} verify the fast decay rate independent of the dimension,
a rigorous proof that the ASPI and the RASPI can achieve the convergence rate we obtained in Section 3 by the best N approximation is still an open question.



Here is the structure of the paper. In Section \ref{background}, we introduce the VFP equation with random electric field. In Section \ref{theoretical}, we prove the best N approximation converges to the solution with an error  of $O\(N^{-s}\)$, for $s>\frac{1}{2}$, based on the result in \cite{cohen2015approximation}. We then in Section \ref{numerical} give an improved numerical method RASPI based on the ASPI introduced in \cite{chkifa2014high}, and provide explicitly the computational cost it saved compared to the ASPI.  Numerical experiments are conducted in Section \ref{numerical eq} to show the convergence rates for various electric fields.  The paper is concluded in Section \ref{Concl}.

\begin{gallarynotation}
\emph{
Define $\O= [0,l]\times\R$ to be the domain for $x, v$. The following norms are defined in $\O$:
\begin{itemize}
\item [-] $\displaystyle \ll h \rl^2 = \int_\O h^2 \,dxdv$.
\item [-] $\displaystyle \ll h \rl_\o^2 = \ll h \rl^2 + \ll \pt_v h\rl^2 + \ll vh \rl^2$.
\item [-] $\displaystyle \ll h \rl_V^2 = \ll h \rl^2 + \ll \pt_x h\rl^2$.
\item [-] $\displaystyle \ll h \rl_{V,\o}^2 = \ll h \rl_\o^2 + \ll \pt_x h\rl_\o^2$.
\end{itemize}
For the metric space $V$, accordingly, one has the following Poincare inequality, 
\begin{align}
	\ll \pt_xh \rl^2 \geq C_s\ll h \rl_V^2,\qd C_s\leq\frac{1}{2},\qd \text{for } \forall h \in V\qd \text{and } \int h dx = 0.
	\label{sobolev const}
\end{align}
Define $U$ as the parameter space for $\bz$, and assume $\bz$ is a random vector with probability density function $\rho(\bz)$, so it has a corresponding weighted $L^2$ norm in the parameter space,
\begin{align}
	\ll h \rl^2_{L^2(U,d\rho)}= \int_U h^2\, d\rho = \int_U h^2\rho(\bz) \,d\bz.
\end{align}
The following norms are defined in $\O\times U$:
\begin{itemize}
\item [-] $\displaystyle \ll h \rl_{L^\infty(U,V)} =\sup_{\bz\in U}\ll h(\bz) \rl_V $;
\item [-] $\displaystyle \ll h \rl_{L^\infty(U,W^{1,\infty}_x)} =\sup_{\bz\in U}\(\ll h(\bz) \rl_{L^\infty_x} + \ll \pt_xh(\bz) \rl_{L^\infty_x}\) $;
\item [-] $\displaystyle \ll h \rl^2_{L^2(U,V,d\rho)} =\int \ll h(\bz) \rl^2_V  d\rho$.
\end{itemize}
Define $\mF$ to be the set of all sequences $\nu = (\nu_j)_{j\geq1}$ of nonnegative integers such that only finite many $\nu_j$ are non-zeros. We call $\L\subset \mF$ an index set. The following notations are defined for index $\nu$:
\begin{itemize}
\item [-] $\displaystyle J(\nu) = \sup\{j: \nu_j >0\}$, \qd $\displaystyle J(\L) = \sup\{J(\nu): \nu \in\L\}$;
\item [-] $\displaystyle |\nu| = \sum_{j\geq1}\nu_j$;
\item [-] $\displaystyle \nu! = \prod_{j\geq1}\nu_j!$;
\item [-] $\displaystyle \pt^\nu_\bz h = \pt^{\nu_1}_{z_1} \cdots \pt^{\nu_n}_{\bz_n}h$, for $n = J(\nu)$;
\item [-] For infinite dimensional vector $d = (d_1, d_2, \cdots)$, $\displaystyle d^\nu = \prod_{j\geq1}d_j^{\nu_j}$.
\end{itemize}
}
\end{gallarynotation}


\section{The parametric Vlasov-Fokker-Planck equation}
\label{background}
Consider the following parametric  Vlasov-Fokker-Planck equation,
\begin{align}
\e\pt_t f + v\pt_xf - \pt_x\p\pt_v f = \frac{1}{\e}\mQ f, \qd x,v \in \O = (0,l) \times \R
\label{VFP}
\end{align}
for $\e\leq 1$ (without loss of generality), with periodic condition on $x\in [0,l]$, and initial data $f(0,x,v,z) = f_0(x,v,z)$.  Here $f$ represents the probability density distribution of particles at position $x$ with velocity $v$, $\e$ represents the rescaled mean free path.  $\mF$ is the Fokker Planck operator that reads 
\begin{align}
	\mQ f = \pt_v\(M\pt_v\(\frac{f}{M}\)\),
\end{align} 
with  the global Maxwellian $M$,
\begin{align}
	M(v) =\frac{1}{\sqrt{2\pi}}e^{-\frac{|v|^2}{2}}. \label{def of M}
\end{align}
$\p(t,x,\bz)$ is  a given parametric potential that reads,
\begin{align}
	\p(t, x,\bz) = \bar{\p}(t, x) + \sum_{j\geq1}z_j \p_j(t, x),
\end{align}
and $E(t, x,\bz) = -\pt_x\p$ is the parametric electric field  that reads,
\begin{align}
	E(t, x,\bz) = \bE(t, x) + \sum_{j\geq1}z_j E_j(t, x) =  \pt_x\bar{\p}(t, x) + \sum_{j\geq1}z_j \pt_x\p_j(t, x).
\label{def of E}
\end{align} 
Here $\bz \in U = [-1,1]^\infty$ is an infinite dimensional parameter.  We
assume assume $\bz$ to be i.i.d random variable following the uniform distribution on $U$, although other distributions can also be used. $\bar{\p}(t,x)$ and $\bar{E}(t,x)$ are the expectations of $\phi, E$ respectively. We furthermore define the corresponding weighted norm $L^2(U,d\rho)$ ,
\begin{align}
	d\rho = \bigotimes_{j\geq1}\frac{dz_j}{2}.
\end{align}

In addition, assume $\p(t,x,\bz)$ and $E(t, x, \bz)$ converge to $\p^\infty(x,\bz)$ and $E^\infty(x,\bz)$ respectively uniformly in time. That is, for $\forall \e >0$, there exists $T$, such that 
\begin{align}
	\sup_{ \bz\in U}\ll \p(t, x,\bz)  - \phi^\infty(x,\bz) \rl_{L^\infty_x} \leq \e, \qd \sup_{ \bz\in U}\ll E(t, x,\bz)  - \Einf(x,\bz) \rl_{L^\infty_x} \leq \e, \qd \text{for } \forall t \geq T.
\end{align}

It is easy to check that 
\begin{align}
	F(x,v,\bz) = e^{-\p^\infty}M(v) \label{def of F}
\end{align}
is the stationary solution of (\ref{VFP}), also called the global equilibrium. Let 
\begin{align}
	h(t,x,v,\bz) = \frac{f(t,x,v,\bz)-F(x,v,\bz)}{\sqrt{M(v)}}
\end{align}
be the {\it{perturbative}} distribution function around F, and furthermore define the perturbative density and perturbative flux as follows, 
\begin{align}
	\s = \int h\sM \, dv, \qd u = \int hv\sM \, dv.
\end{align}
Furthermore, in this paper, we will only focus on the randomness that comes from the electric random field. Therefore, we assume the following condition on the initial data.
\begin{condition}
\label{ass: initial}
Assume there is no initial random perturbation around the steady state F(x,v,z), and the initial perturbative mass is zero. That is, the initial data satisfies the following two equations:
\begin{align}
	&f(0,x,v,z) = F(x,v,z) + h(0,x,v)\sM(v); \label{indep of z}\\
	&\int h(0,x,v)\sM(v) dxdv = 0. \label{initial cond}
\end{align}
\end{condition}

It is easy to check that $h$ satisfies,
\begin{align}
	\e\pt_th + v\pt_xh - \frac{1}{\e} \mL h = -E\(\pt_v - \frac{v}{2}\)h + v\sM \(E - E^\infty\)e^{-\p^\infty},
	\label{micro eq with t}
\end{align}
where $\mL$ is the linearized Fokker-Planck collision operator,
\begin{align}
	\mL h = \frac{1}{\sM} \pt_v\(M \pt_v\(\frac{h}{\sM}\)\), \label{linearized FP}
\end{align}
which satisfies the local coercivity property \cite{duan2010kinetic}, 
\begin{align}
	-\la \mL h, h \ra \geq \lam\ll (1-\Pi)h \rl^2_\o,\label{coercivity}
\end{align}
where $\lam < 1$ is a constant, $\Pi$ is the projection operator onto the null space of $\mL$, 
\begin{align}
	\Pi h = \( \int_{\R} h \,dv\) \sM.
\end{align}
We call equation (\ref{micro eq with t}) the {\it{microscopic equation}}.

\section{Decay rate of the best N approximation}
\label{theoretical}
In this section, we first review the best N approximation, and then
study the convergence rate of this method applied to the Vlasov-Fokker-Planck equation with random electric field.

\subsection{The best N approximation and our result}
\label{sec: best n result}
Since the solution $f(t,x,v,\bz) = F + \sM h(t,x,v,\bz)$, where $F$ and $M$ are given in (\ref{def of F}) and (\ref{def of M}) respectively, as long as one gets the approximate solution for h, then one can easily obtain the approximation solution for $f$. Hence we  seek approximate solution $h_\L$ in a finite dimensional space, 
\begin{align}
	\P_\L = \{h_\L :  h_\L  = \sum_{\nu \in \L} h_\nu(t,x,v)L_\nu(\bz)\},
\end{align}
where $\L$ is an index set with infinite dimensional vectors $\nu$. Here $L_\nu(\bz)$ is the orthonormal Legendre polynomial which forms a basis in $L^2(U,d\rho)$ such that, 
\begin{align}
	L_\nu = \prod_{j\geq1} L_{\nu_j}(z_j), \qd \int_{-1}^1 L_k(z_j)L_l(z_j) \frac{dz_j}{2}  = \d_{kl}. 
\end{align}

If $h$ solves (\ref{micro eq with t}), then naturally one has the projection of the solution $h$ onto $\P_\L$, 
\begin{align}
	P_\L h = \sum_{\nu\in \L} \(\int_U h L_\nu d\rho\) L_\nu := \sum_{\nu\in \L}  h_\nu  L_\nu = \argmin_{h_\L \in \P_\L} \ll h - h_\L \rl_{L^2(U,V,d\rho)}.
\end{align}
The best N approximation is a form of nonlinear approximation that searches for $\nu\in \L$ according to the largest $N$ coefficients $\ll h_\nu \rl_V$. It is proved in \cite{cohen2015approximation} that the decay rate of such approximation  depends on the holomorphy and anisotropy of the solution in the random space, as stated in the following theorem.
\begin{theorem}
 [Corollary 3.11 of \cite{cohen2015approximation}]Consider a parametric equation of the form 
\begin{align}
	\mP(f,a) = 0,
\end{align}
with random field $a = \bar{a}(x) + \sum_{j\geq1}z_j\psi_j(x) \in X$, $\bz\in U$, where $X$ is certain space of $x$. Assume the solution map $a\to f(a)$ admits a holomorphic extension to an open set $\mO\in X$ which contains $a(U) = \{a(\bz): \bz\in U\}$, with uniform bound
\begin{align}
	\sup_{a\in \mO}\ll f(a) \rl_V \leq C.
\end{align}
If in addition $\(\ll \psi_j \rl_X\)_{j\geq1} \in \ell^p(\N) $ for some $p<1$, then for the set of indices $\L_n$ that corresponds to the $n$ largest $f_\nu = \ll \int f L_\nu d\rho \rl_V$, one has,
\begin{align}
	\ll f - \sum_{\nu\in\L_n} f_\nu L_\nu \rl_{L^2(U,V,d\rho)} \leq \frac{C_p}{(n+1)^s}, \qd s= \frac{1}{p} - \frac{1}{2}
\end{align}
where $C_p:=\ll \(\ll f_\nu\rl_V\)_{\nu\in\mF} \rl_{\ell^p}$.
\end{theorem}

In order to apply the above theorem, one needs to prove  the holomorphy of the solution map. In the case of kinetic equation, since we are only dealing with a real function, the holomorphy of a solution map is equivalent to: There exists constant $B$ and $C_j$, such that
\begin{align}
	\ll \pt_{z_j}^{k} h \rl_V^2 \leq B C_j^{k}k!, \qd\text{for any nonnegative integer }k. \label{NTP}
\end{align}
This result will be proved in Theorem \ref{pt_nu h of t}. In the following context, $\nu$ always represents an infinite dimensional index, $e_j$ is an infinite dimensional vector with only the $j$-th component being 1 and all others  zeros. $\pt^\nu$ represents $\pt^\nu_\bz$ for the convenience of writing. Theorems \ref{pt_nu h of t} and \ref{conv rate of t} are  both based on the following assumptions on $E(t, x,\bz)$ and $\{E_j(t, x)\}_{j\geq 1}$. 
\begin{condition}
\label{cond of E of t}
Assume $E(t,x,\bz) = \bar{E}(t,x) + \sum_{j\geq1}z_jE_j(t,x)$ satisfies the following assumptions,
\begin{align}
&\sup_{t\geq0}\ll E_j(t) \rl_{L^\infty\(U,W^{1,\infty}_x\)} \leq C_j, \qd \(C_j\)_{j\geq1} \in l^{p}(\N), \text{ for some }p\leq 1\label{anisotropic}\\
& \ll E (t)\rl_{L^\infty(U,W^{1,\infty}_x)} \leq \ll\bar{E}(t) \rl_{W^{1,\infty}_x}  +  \sum_{j\geq1}\ll E_j(t) \rl_{W^{1,\infty}_x} \leq C_E, \text{ for all }t,\qd C_E  \leq \frac{\lam C_s}{8}.\label{cond of C_j}
\end{align}
There exists a continuous function $D(t)$, such that 
\begin{align}
\ll \pt^\nu\(\(E - \einf\)e^{-\pinf}\) \rl^2_{L^\infty(U,V)} \leq D(t)(K^\nu\nu!)^2, \text{ for all }t,\qd \int_0^\infty D(s)dt = \bar{D},\label{def of barD}
\end{align}
where $\(K_j\)_{j\geq1} \in l^{p}(\N)$,  for some $p\leq 1$.

\end{condition}

\begin{remark}
If one uses the Karhunen-Lo\`eve expansion to parametrize the random field, then the smoothness properties of the covariance function for the random field determine the $\ell_p-$summability of the random variables. For random field $a(\o, x)\in L^2(\O, dP; L^\infty(D))$ in a polyhedral domain $D\subset \R^d$ with mean field $E_a(x) = \int_\O a(\o, x) dP(\o)$ and covariance $V_a(x, x') = \int_\O \(a(\o,x) - E_a(x)\)\(a(\o, x') - E_a(x')\) dP(\o)$, if the stationary covariance $g_a(z)$ is analytic outside of $z=0$, and $C^k$ at zero, where $V_a(x, x') = g_a\(\lv x - x'\rv\)$, then it is $\ell^{d/k}$-summable.  See \cite{schwab2006karhunen, todor2006robust} for details.
\end{remark}

In the above assumptions, equation (\ref{anisotropic}) guarantees the anisotropy of $E$. (\ref{cond of C_j}) - (\ref{def of barD}) are required for the analyticity of the solution in the random space. Basically, it requires $E$ to be bounded and converges to $\Einf$ fast enough so that the improper integral $D(t)$ exists.

We first state our Theorem about the analyticity of $h$.
\begin{theorem}
\label{pt_nu h of t}
Under Assumptions  \ref{ass: initial} and \ref{cond of E of t} the $\nu$-th derivative of the perturbative solution to (\ref{micro eq with t}) in the random space can be bounded as follows,
\begin{align}
	\ll \pt^\nu h(t) \rl_{L^\infty(U,V)} \leq Q(t)\(|\nu|!\)d^\nu,
	\label{final est_1 of tt}
\end{align}
where $d$ is an infinite dimensional vector with the j-th component 
\begin{align}
	d_j = \max\l\{\sqrt{\frac{20C_j}{\lam C_s}}, K_j\r\}, \label{def d}
\end{align} 
where $C_j, K_j$ are defined in (\ref{anisotropic}), (\ref{def of barD}) respectively; $Q(t)$ is a function exponentially decaying  in $t$, 
\begin{align}
	Q(t) =  \min\l\{\frac{1}{\e} e^{-\frac{\xi}{\e^2}t}, e^{-\xi t}\r\}2\(\ll h(0) \rl_V+ \sqrt{\frac{\lam \bar{D}}{C_E}}\),
	\label{def Q}
\end{align}
for $\xi = \frac{ \lam C_s}{10}$, $C_s, \lam, \bar{D}$ are constants defined in (\ref{sobolev const}), (\ref{coercivity}), (\ref{def of barD}) respectively.
\end{theorem}

By taking $\nu = ke_j$, one can derive the inequality (\ref{NTP}), which implies the analyticity of $h$ in the random space. Therefore, based on the above theorem and assumptions, we can conclude that the best N approximation converges independent of dimensionality of the parameter $\bz$ and faster than the Monte Carlo method, as stated in the following theorem:
\begin{theorem}
\label{conv rate of t}
If  $E(t, x,\bz)$ satisfies Condition \ref{cond of E of t}, then the approximate solution obtained by the best N approximation converges to the exact solution with the error,
\begin{align}
	\ll f - f_\L \rl_{L^2(V,U,d\rho)} \leq \frac{C_p}{(n+1)^s}, \qd s= \frac{1}{p} - \frac{1}{2},
	\label{final est_11 of t}
\end{align}
where $p\leq1$ depends on (\ref{cond of C_j}), $C_p = \ll (\ll f_\nu \rl_V )_{\nu\in\mF}\rl_{\ell^p} <\infty$.
\end{theorem}

\begin{remark}
How large is $C_p$?
\begin{itemize}
\item [-] For the case when $\ll d/\sqrt{3} \rl_{\ell^1} < 1$:\\
\cite{cohen2010convergence} gives a way to calculate the upper bound for $C_p$ when $\ll d/\sqrt{3} \rl_{\ell^1} < 1$. First by Rodrigre's Formula (See Section 6 of \cite{cohen2010convergence}), 
\begin{equation*}
	\ll h_\nu\rl_V = \ll \int h L_\nu d\rho \rl_V  = \frac{\(\sqrt{3}\,\)^{-\nu}}{\nu!}\ll \pt^\nu h \rl_{L^\infty(U,V)}.
\end{equation*}
Then by the estimate in (\ref{final est_1 of tt}), one has
\begin{equation*}
	\ll h_\nu\rl_V \leq  B(t)\frac{|\nu|!}{\nu!}\(\frac{d}{\sqrt{3}}\)^\nu.
\end{equation*}
According to Theorem 7.2 of \cite{cohen2010convergence}, let $\a = \frac{d}{\sqrt{3}}$ be an infinite dimensional vector, then if $\ll \a \rl_{\ell^1} \leq 1$, and $\a\in \ell^p$, one has
\begin{equation}
	\ll (\ll h_\nu \rl_V )_{\nu\in\mF}\rl_{\ell^p}  \leq \frac{2}{\eta}\exp\(\frac{2(1-p)(J(\eta)+\ll \a \rl_{\ell^p}^p)}{p^2\eta}\)
\end{equation}
where $\eta = \frac{1-\ll \a \rl_{\ell^1}}{2}$, $J(\eta)$ is the smallest positive integer such that $\sum_{j \geq J}|\a_j|^p\leq \frac{\eta}{2}$.

\item [-] For the case $\ll d/\sqrt{3} \rl_{\ell^1} \geq1$:\\
There is no explicit expression for the upper bound of $C_p$ (See Remark 3.22 of \cite{cohen2015approximation}).
\end{itemize}
\end{remark}

\subsection{Proof of Theorem \ref{pt_nu h of t}}
\label{proof of best n}
Define a Lyapunov functional 
\begin{align}
&G^\nu_i= \theta_i\(\frac{\e}{2} \ll \pt^{\nu}h \rl_V^2\) +\frac{1}{2\e} \(\e\la \pt^{\nu}u, \pt_x\pt^{\nu}\s \ra + \frac{1}{2}\ll \pt^{\nu}\s\rl^2\), \qd i = 1,2.
\label{def of G}
\end{align} 
with $\theta_1 = \frac{8}{7\lam},\qd \theta_2 = \frac{8}{7\lam\e^2}$.
Similar Lyapunov functional has been introduced in \cite{hwang2013vlasov} for the deterministic nonlinear Vlasov-Poisson-Fokker-Planck (VPFP) system with $\e=1$. For the case where the uncertainty and scaling parameter $\e$ are involved, \cite{ZhuJin18} gives a modified Lyapunov functional, which is more suitable for different scaling of $\e$. Actually, $G^\nu_i$ is equivalent to $\ll \pt^\nu h\rl_V^2$. Since by Young's inequality, one has
\begin{align}
	 -\frac{\e^2}{2}\ll \pt^{\nu} \pt_xu\rl^2 -\frac{1}{2} \ll \pt^{\nu}\s \rl^2 \leq \e\la\pt^{\nu} \pt_xu, \pt^{\nu}\s \ra
	 \leq & \frac{\e^2}{2}\ll \pt^{\nu} \pt_xu\rl^2 +\frac{1}{2} \ll \pt^{\nu}\s \rl^2,\nonumber
\end{align}
so, 
\begin{align}
	 -\frac{\e}{4}\ll \pt^{\nu} \pt_xu\rl^2  \leq \frac{1}{2\e}\(-\e\la\pt^{\nu} \pt_xu, \pt^{\nu}\s \ra + \frac{1}{2}\ll \pt^{\nu}\s\rl^2\) 
	 \leq &  \frac{\e}{4}\ll \pt^{\nu} \pt_xu\rl^2 +  \frac{1}{2\e}\ll \pt^{\nu}\s\rl^2 .\nonumber
\end{align}
Because that $\ll u \rl^2, \ll \s \rl^2 \leq \ll h \rl^2$, the above inequality becomes
\begin{align}
	-\frac{\e}{4}\ll \pt^{\nu} h\rl_V^2 \leq \frac{1}{2\e}\(-\e\la\pt^{\nu} \pt_xu, \pt^{\nu}\s \ra + \frac{1}{2}\ll \pt^{\nu}\s\rl^2\) 
	\leq \frac{1}{2\e}\ll \pt^\nu h \rl^2_V .
\end{align}
Plug the above inequalities to the definition of $G^\nu_i$, one ends up with
\begin{align}
	&\text{for }\th_1 = \frac{8}{7\lam}, \qd \frac{\e}{2\lam} \ll \pt^{\nu}h \rl_V^2  \leq G_1^\nu \leq \frac{3}{2\lam\e} \ll \pt^{\nu}h \rl_V^2 \\ 
	&\text{for }\th_2 = \frac{8}{7\lam\e^2}, \qd \frac{1}{2\lam\e} \ll \pt^{\nu}h \rl_V^2  \leq G_2^\nu \leq \frac{3}{2\lam\e} \ll \pt^{\nu}h \rl_V^2.
\end{align}

\begin{lemma}
\label{energy est of t}
Under Condition \ref{cond of E of t}, for  $\forall \bz \in U$,  the following estimates hold, 
\begin{align}
	&\text{for }|\nu| = 0:  \pt_t G^0_i  + \frac{\eta}{\e}\ll h\rl_V^2  \leq  \frac{1}{\e C_E}\ll \pt^\nu\(\(E - \einf\)e^{-\pinf}\) \rl^2_V;\label{G of t}\\
	&\text{for }|\nu| > 1: \pt_t G^\nu_i +  \frac{\eta}{\e}\ll h\rl_V^2\leq\frac{2}{\lam \e} \sum_{\nu_j\neq0} \nu_j^2C_j\ll \pt^{\nu-e_j}h \rl_V^2
	+ \frac{1}{\e C_E}\ll \pt^\nu\(\(E - \einf\)e^{-\pinf}\) \rl^2_V,\label{G_nu of t}
\end{align}
where $i = 1,2$, $\eta = \frac{C_s}{10}$, and 
\begin{align}
	&\frac{\e}{2\lam} \ll \pt^{\nu}h \rl_V^2  \leq G_1^\nu \leq \frac{3}{2\lam\e} \ll \pt^{\nu}h \rl_V^2,\qd \frac{1}{2\lam\e} \ll \pt^{\nu}h \rl_V^2  \leq G_2^\nu \leq \frac{3}{2\lam\e} \ll \pt^{\nu}h \rl_V^2 .\label{G_2 of t}
\end{align}
\end{lemma}
\begin{proof}
See Appendix \ref{proof of energy est of t}.
\end{proof}

\begin{lemma}
\label{lemma: final est_1 of t}
For fixed $\bz$, the following estimates hold,
\begin{align}
	G_i^\nu(t) \leq\(|\nu|!\)^2(2d)^{2\nu} \(G^0_i(0) + \frac{\bar{D}}{\e C_E}\)- \frac{\eta}{\e}\int_0^t \ll \pt^\nu h(s) \rl^2_V ds,
	\label{final est_1 of t}
\end{align}
where $d=(d_1, d_2, \cdots)$ with $d_j = \max\l\{ \sqrt{\frac{2C_j}{\lam\eta}}, K_j\r\}$.
\end{lemma}

\begin{proof}

First for $| \nu | = 0$, by (\ref{G of t}) in  Lemma \ref{energy est of t},  one has 
\begin{align}
	&G^0_i(t) \leq G^0_i(0) + \frac{1}{\e C_E}\int_0^tD(s)ds- \frac{\eta}{\e}\int_0^t \ll h(s) \rl^2 ds,
	\label{final est_1}
\end{align}
which satisfies (\ref{final est_1 of t}). Then by induction, assume the following holds, 
\begin{align}
	G_i^{\nu-e_j}(t) \leq\(\(|\nu| - 1\)!\)^2(2d)^{2(\nu-e_j)} \(G^0_i(0) + \frac{\bar{D}}{\e C_E}\) - \frac{\eta}{\e}\int_0^t \ll \pt^{\nu-e_j} h(s) \rl^2_V ds.
	\label{nu-ej of t}
\end{align}
By integrating (\ref{G_nu of t}) over $t$, one has,
\begin{align}
	G_i^\nu(t) \leq G^\nu_i(0) - \frac{\eta}{\e}\int_0^t \ll \pt^{\nu}h(s)\rl_V^2ds
	+ \frac{ 2}{\lam\e}\sum_{\nu_j\neq0} \nu_j^2C_j \int_0^t\ll \pt^{\nu-e_j}h(s) \rl_V^2ds+ \frac{\bar{D}}{\e C_E}(K^\nu\nu!)^2.\label{pc_21 of t}
\end{align}

Since the initial perturbation $h$ is independent of the parameter $\bz$, so $G_i^{\nu}(0) = 0$ for $\forall |\nu|>0$. Multiplying $\frac{2\nu_j^2C_j}{\lam\eta}$ to (\ref{nu-ej of t}), and summing it over $\nu_j \neq 0$, then combining it with (\ref{pc_21 of t}), one gets
\begin{align}
	G_i^\nu(t) + \sum_{\nu_j\neq0}  \frac{2\nu_j^2C_j}{\lam\eta}G_i^{\nu-e_j}(t) 
	\leq &\sum_{\nu_j\neq0}   \frac{2\nu_j^2C_j}{\lam\eta}\(\(|\nu| - 1\)!\)^2(2d)^{2(\nu-e_j)} \(G^0_i(0) + \frac{\bar{D}}{\e C_E}\)  \nonumber\\
	&+ \frac{\bar{D}}{\e C_E}(K^\nu\nu!)^2- \frac{\eta}{\e}\int_0^t \ll \pt^{\nu}h(s)\rl_V^2ds\,.  \label{336}
\end{align}
Since $G^{\nu-e_j}(t)$ is always positive, we can omit the second term on the LHS. In addition, note that $\frac{2C_j}{\lam\eta}(2d)^{2(\nu-e_j)} \leq (2d)^{2\nu}/2$, so
\begin{align}
	&\frac{\bar{D}}{\e C_E}(K^\nu\nu!)^2 + \sum_{\nu_j\neq0}   \frac{2\nu_j^2C_j}{\lam\eta}\(\(|\nu| - 1\)!\)^2(2d)^{2(\nu-e_j)} \(G^0_i(0) + \frac{\bar{D}}{\e C_E}\)\nonumber\\
	\leq &\frac{\bar{D}}{\e C_E}(d^\nu\nu!)^2 + (2d)^{2\nu} \frac{1}{2}\(G^0_i(0) + \frac{\bar{D}}{\e C_E}\) \(\(|\nu| - 1\)!\)^2\sum_{\nu_j\neq0}   \nu_j^2\nonumber\\
	\leq &\frac{\bar{D}}{\e C_E}(d^\nu\nu!)^2 + \((2d)^\nu |\nu|!\)^2 \frac{1}{2} \(G^0_i(0) + \frac{\bar{D}}{\e C_E}\)\nonumber\\
	\leq &((2d)^\nu |\nu|!)^2\( \frac{1}{2}\(G^0_i(0) + \frac{\bar{D}}{\e C_E}\) +\frac{1}{2}\frac{\bar{D}}{\e C_E}\)
	\leq \((2d)^\nu |\nu|!\)^2 \(G^0_i(0) + \frac{\bar{D}}{\e C_E}\)\label{337}
\end{align}
where the second inequality is because of $\sum_{\nu_j\neq0}\nu_j^2 \leq |\nu|^2$, and the third inequality holds for any $|\nu|>0$. Plugging (\ref{337}) into (\ref{336}), and omitting the second term on the LHS give (\ref{final est_1 of t}) complete the induction and consequently the proof for Lemma \ref{lemma: final est_1 of t}.
\end{proof}

From Lemma \ref{lemma: final est_1 of t}, and  the equivalent relationship between $G_1^\nu$ and $\ll \pt^\nu h \rl^2_V$ in (\ref{G_2 of t}),  (\ref{final est_1 of t}), one has
\begin{align}
	\frac{\e}{2\lam }\ll \pt^\nu h \rl^2_V \leq&\(|\nu|!(2d)^\nu\)^2 \( \frac{3}{2\lam\e}\ll h(0) \rl^2_V + \frac{\bar{D}}{\e C_E}\) - \frac{\eta}{\e}\int_0^t \ll \pt^{\nu}h(s)\rl_V^2ds ,\nonumber\\
	\ll \pt^\nu h \rl^2_V \leq& \(|\nu|!b^\nu\)^2 \(\frac{3}{\e^2}\ll h(0) \rl^2_V +\frac{2\lam\bar{D}}{\e^2 C_E}\) - \frac{2\lam\eta}{\e^2}\int_0^t \ll \pt^{\nu}h(s)\rl_V^2ds.
\end{align}
By Grownwall's inequality 
\begin{align}
	&\ll \pt^\nu h(t) \rl_V \leq \frac{2}{\e} |\nu|!b^\nu\(\ll h(0) \rl_V + \sqrt{\frac{\lam \bar{D}}{C_E}}\) e^{-\frac{\xi t}{\e^2}},
\end{align}
for $\xi = \frac{\lam C_s}{10}$. Similarly, for $G_2^\nu$, by (\ref{G_2 of t}) and (\ref{final est_1 of t}), one obtains
\begin{align}
	& \ll \pt^\nu h(t) \rl^2_V \leq (|\nu|!b^\nu)^2\(3 \ll h(0) \rl^2_V + \frac{2\lam \bar{D}}{C_E}\) - 2\lam\eta\int^t_0 \ll \pt^\nu h(s)\rl_V^2ds.\nonumber
\end{align}
Grownwall's inequality then implies,
\begin{align}
	&\ll \pt^\nu h(t) \rl_V \leq 2|\nu|!b^\nu \(\ll h(0) \rl_V+ \sqrt{\frac{\lam \bar{D}}{C_E}}\)  e^{-\xi t},
\end{align}
which gives the conclusion in Theorem \ref{pt_nu h of t}.

\section{The Numerical Method}
\label{numerical}
The convergence rate obtained in Theorem \ref{conv rate of t} is based on the best N approximation, which means one needs to calculate all coefficients of the Legendre series $h_\nu$ in order to find the N largest $\ll h_\nu\rl_V$. In practice, one needs a more efficient numerical method to find the best basis. Based on the greedy search method introduced in \cite{chkifa2014high}, in line of \cite{nobile2008sparse,nobile2008anisotropic},  we formulate a new {\it{residual based adaptive sparse polynomial interpolation}} (RASPI) method. We will first introduce the framework of the {\it{adaptive sparse polynomial interpolation}} (ASPI) method  in Section \ref{ASPI}. Then in Section \ref{RASPI}
the new residual based method will be introduced, and the reason why this method is computationally efficient,  particularly for time dependent kinetic equation when $\e$ is small, is also explained in Section \ref{RASPI}. Finally a comparison of the  computational cost between the  ASPI and the RASPI methods for general kinetic equations is given in Section \ref{compare}.

In this section, we assume $\ll E_j(t, x) \rl_V$ decreases as $j$ increases for all $t\geq0$.

\subsection{The adaptive sparse polynomial interpolation (ASPI)}
\label{ASPI}

The ASPI is a numerical method that approximates the solution map by a sparse polynomial interpolation at well chosen points. Let us first define the representation of infinite dimensional random variable and polynomial interpolation bases. For a sequence $\G = (\b_k)_{k\geq0}$ of distinct points in $[-1,1]$, and index $\nu = (\nu_j)_{j\geq 1}$, define points 
\begin{align}
	\bz_\nu = (\b_{\nu_j})_{j\geq1}
\end{align} 
and hierarchical Lagrange basis 
\begin{align}
	H_\nu(\bz) = \prod_{j\geq1}l_{\nu_j}(z_j), \qd l_0 = 1,\qd l_k(\b) = \prod_{m=0}^{k-1}\frac{\b - \b_m}{\b_k - \b_m}. \label{hierarchical basis}
\end{align}
Note that 
\begin{align}
	&H_\nu(\bz_\nu) = 1, \qd \text{for all }\nu\in\mF, &H_\nu(\bz_{\t{\nu}}) = 0, \qd \text{for all }\t{\nu}<\nu. \label{H_nu}
\end{align}
Here $\t{\nu} \leq \nu$ if and only if all components of $\t{\nu}$ are smaller than or equal to $\nu$; $\t{\nu}<\nu$ represents that $\nu\leq \nu$ and $\t{\nu} \neq \nu$.
 We call the index set $\{\L_k\}_{k\geq1}$ {\it{monotone}} if $\L_k \subset \L_{k+1}$ for all $k$. We further call index set $\L\subset \mF$ {\it{downward closed}}, 
\begin{align}
	\text{if }\nu\in\L, \qd\t{\nu}\leq \nu,\qd \text{then } \t{\nu} \in\L.
\end{align}

Secondly, when is the infinite dimensional polynomial interpolation well defined? Actually for a downward closed set $\L \subset \mF$, given the grid $\bz_\L$ and the corresponding solution $f_\L$ on the grids,
\begin{align}
	\bz_\L := \{ \bz_\nu, \nu\in\L\}, \qd f_\L : = \(f_\nu\)_{\nu\in\L} := \(f(t,x,v,\bz_\nu)\)_{\nu\in\L},
\end{align}
there exists a unique polynomial 
\begin{align}
	I_\L (t,x,v,\bz)= \sum_{\nu\in\L} \a_\nu(t,x,v)H_\nu(\bz), 
\end{align}
such that $I_\L$ has the same value as $f_\L$ at $\bz_\L$. Namely, $I_\L$ is the polynomial interpolation of $f_\L$ at interpolating points $\bz_\L$. From the above framework, the multi-dimensional polynomial interpolation is uniquely determined by the sequence $\G$ and index set $\L$. There are three questions to be answered at this point. 
\begin{itemize}
\item How to choose the sequence $\G = (\b_k)_{k\geq0}$;
\item How to calculate $I_\L$ if given $\bz_\L$ and $f_\L$;
\item How to find the $\L_n$ with $\#(\L_n) = n$, such that $I_{\L_n}$ is the closest to $f(t,x,v,\bz)$,
\end{itemize}
where $\#(\L_n)$ represents number of elements  in $\L_n$.

Choosing different sequences will result in different stability and accuracy of the interpolation mapping, which is characterized by the Lebesgue constant. The Leja sequence is usually considered  a good choice, which starts with an arbitrary $\b_0\in [-1,1]$, and then defined by,
\begin{align}
	\b_k:=\text{argmax}\l\{\prod_{l=0}^{k-1}\lv \b-\b_l\rv: \b\in[-1,1] \r\}.\label{leja_const}
\end{align}  
\cite{chkifa2014high} proved that if the Lebesgue constant of  a univariate polynomial interpolation on sequence $\{\b_l\}_{l=0}^{k}$ is $O\((k+1)^\theta\)$, then the Lebesgue constant $\lam_\L$ of polynomial interpolation on $\bz_\L$ is $O(\#(\L)^{\theta+1})$ for any monotone set $\L$. \cite{chkifa2013lebesgue} proved the Lebesgue constant on the Leja sequences is less than $3(k+1)^2\log(k+1)$, which implies the Lebesgue constant of the multidimensional polynomial interpolation on $\bz_\L$ is less than $O\(\#(\L)^4\)$.

After determining the sequence $\G$, given arbitrary $\L_n = \{\nu_1,\cdots, \nu_n\}$, and the corresponding $\bz_{\L_n}, f_{\L_n}$, since the interpolation polynomial satisfies
\begin{align}
	\begin{bmatrix}  
		H_{\nu_1}(\bz_{\nu_1})\cdots H_{\nu_n}(\bz_{\nu_1})\\
		\vdots\\
		H_{\nu_1}(\bz_{\nu_n})\cdots H_{\nu_n}(\bz_{\nu_n})\\
	\end{bmatrix}
	\begin{bmatrix}  
		\a_{\nu_1}\\
		\vdots\\
		\a_{\nu_n}\\
	\end{bmatrix}
	=\begin{bmatrix}  
		f_{\nu_1}\\
		\vdots\\
		f_{\nu_n}\\
	\end{bmatrix},
\end{align}
one can invert the first matrix to get the coefficient $\(\a_\nu\)_{\nu\in\L_n}$. In general, one needs to do the inversion all over again if the index $\L_n$ changes.  

However, if $\{\L_n\}_{n=1}^N$ is monotone and downward closed, there is a progressive construction of the interpolation operator, which allows to avoid
inverting a matrix. If $\L_n = \L_{n-1} \cup \{\nu_n\}$, then 
\begin{align}
	I_{\L_n} = I_{\L_{n-1}} + \a_{\nu_n} H_{\nu_n},\qd \a_{\nu_n} = f_{\nu_n} - I_{\L_{n-1}}\(\bz_{\nu_n}\), \qd \text{with }I_{\L_0} = 0.\label{poly inter}
\end{align}
Actually, one can prove this by induction. For $n = 1$, $I_{\L_1} = f_{\nu_1}$ is indeed the interpolation on $\bz_{\nu_1}$. Assume $I_{\L_{n-1}}$ constructed in the above way is the interpolation on $\bz_{\L_{n-1}}$, then since $\nu_1 <  \cdots <\nu_{n-1} < \nu_n$, so by (\ref{H_nu}), $H_{\nu_n}(\bz_{\nu_k}) = 0$, for $k\leq n-1$; $H_{\nu_n}(\bz_{\nu_n}) = 1$. Therefore 
\begin{align}
	&k\leq n-1: \qd I_{\L_n}(\bz_{\nu_k}) = I_{\L_{n-1}}(\bz_{\nu_k}) = f_{\nu_k},\\
	&k= n: \qd I_{\L_n}(\bz_{\nu_n}) = I_{\L_{n-1}}(\bz_{\nu_n}) + (f_{\nu_n} -I_{\L_{n-1}}(\bz_{\nu_n}))= f_{\nu_n},
\end{align}
which implies that $I_{\L_n}$ is the interpolation operator on $\bz_{\L_n}$.

In order to use this progressive construction to find the interpolation operator, we require the index set $\{\L_n\}_{n=1}^N$ to be monotone and downward closed, that is, 
\begin{align}
	\L_{n+1} = \L_n \cup \{\nu_{n+1}\}, \qd \nu_n \in N(\L_n), \qd N(\L_n) = \{\nu\notin \L_n, \L_n\cup\{\nu\} \text{ is downward closed}\}\nonumber
\end{align}
where we call $N(\L_n)$ the neighborhood of index set $\L_n$.

Now we come to the last question. Assume we already determined $\L_n$, in order to find the best $\L_{n+1}$,  how should one select the optimal $\nu_{k+1}$ from the neighborhood of $\L_n$? First we notice for infinite dimensional $\bz$, $\#\{N(\L_k)\}$ is also infinite. Even for finite dimension $\bz\in \R^d$, $\#\{N(\L_k)\} \sim O(k^d)$, which is too big to search numerically. So we introduce {\it{anchored neighbors}} $\t{N}(\L)$, 
\begin{align}
	\t{N}(\L) = \{\nu\in N(\L):  \nu_j =0 \text{ if } j < j(\L)+1\}, \qd j(\L) = \max\{j: \nu_j > 0\text{ for some }\nu\in \L\}
\end{align}
The reason why searching the anchored neighbor makes sense is because we assume at the beginning of this section $\ll E_j(t,x)\rl_{W^{1,\infty}_x}$ decreases as $j$ increases, then from Theorem \ref{pt_nu h of t}, one notices the upper bound of $\ll \pt_{z_j}h(t,\bz) \rl_V$ decreases as $j$ increases, which formally indicates that $z_j$ becomes less sensitive when $j$ increases. So if for all $\nu \in \L_n$, the components larger than and equal to $\(j(\L_n)+1\)$ of $\bz_\nu$ are the same,  then when searching for the next interpolation point, one should first consider adding a point along $\(j(\L_k)+1\)$-st component before all the other components larger than $\(j(\L_k)+1\)$. 

Note that because of the monotonicity of $\L_n$, one can actually construct $\t{N}(\L_n)$ based on $\t{N}(\L_{n-1})$ in the following way. Define
\begin{align}	
 	&\h{N}(\L_n) = \{e_{j(\L_n)+1}, \nu_n+e_j, j \leq j(\nu_n)\}, \qd j(\nu) = \max\{j: \nu_j > 0\}\label{N_Ln}\\
	&N^*(\L_n) =  \h{N}(\L_n) \backslash \(\t{N}(\L_{n-1}) \cap \h{N}(\L_n)\), \label{def N*}
\end{align}
then 
\begin{align}
	&\t{N}(\L_n) = \t{N}(\L_{n-1}) \cup N^*(\L_n). \label{def of NLn}
\end{align}

Here is an example that shows the anchored neighbors in three dimension.
\begin{figure}[htbp]
\subfloat[]{\includegraphics[width=0.33\textwidth]{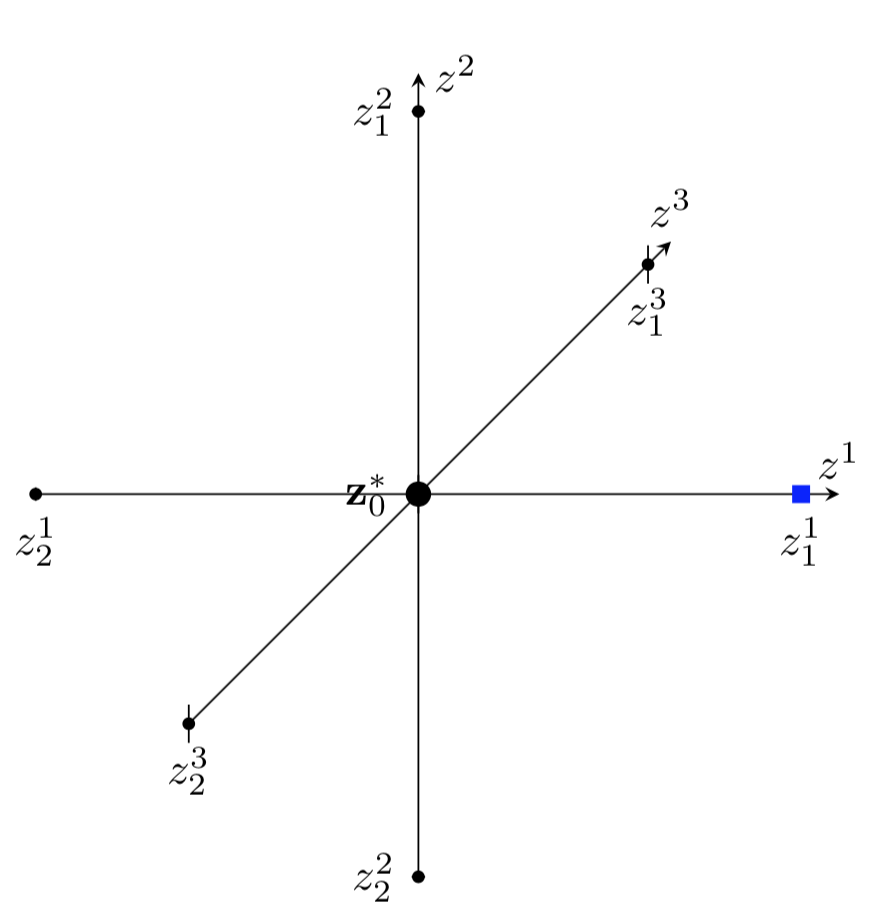}}
\subfloat[]{\includegraphics[width=0.33\textwidth]{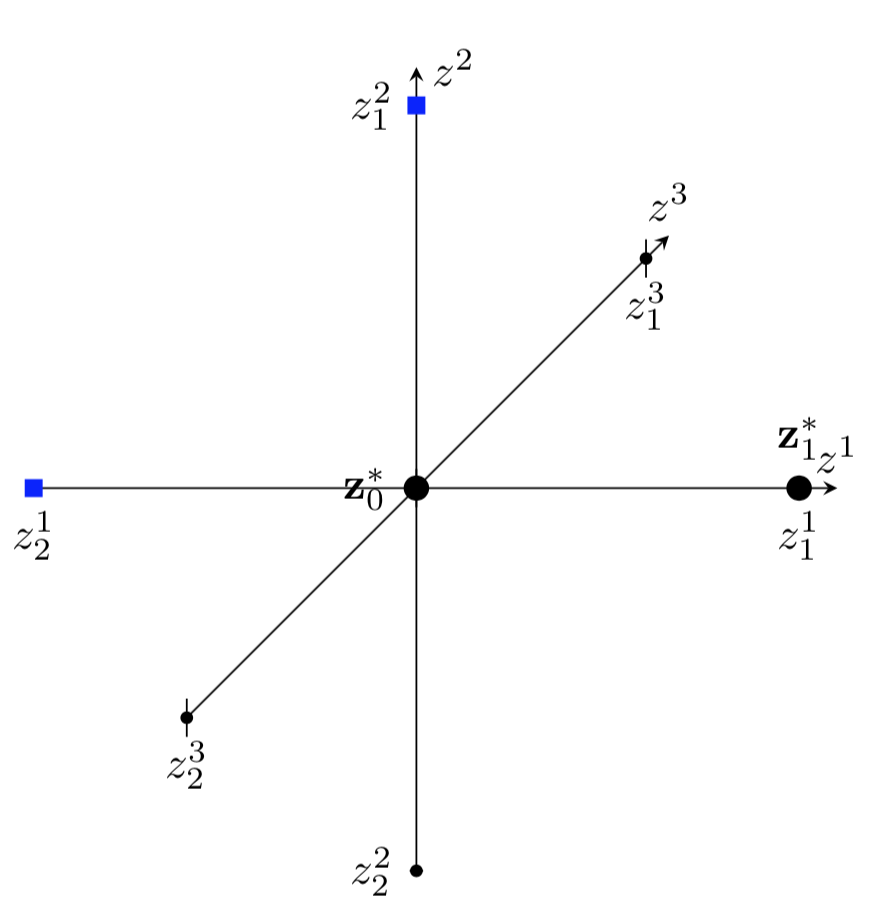}}
\subfloat[]{\includegraphics[width=0.33\textwidth]{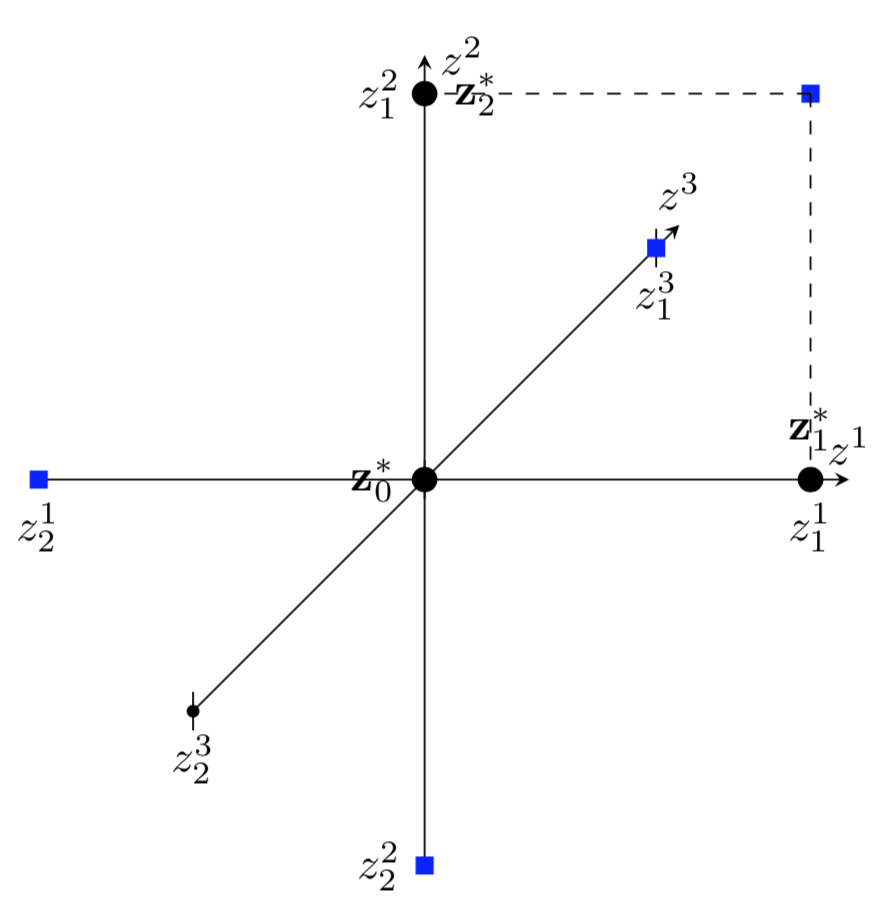}}
\caption{The blue dots represent the anchored neighbors of $\L_0 = \{\bz_0^*\}, \L_1 =\{\bz_0^*, \bz_1^*\}, \L_2 = \{\bz_0^*, \bz_1^*, \bz_2^*\}$ for $\bz= (z^1,z^2,z^3)\in[-1,1]^3$. }
\label{fig: 3D}
\end{figure}
Since we assume the direction $z^1$ is more important than $z^2, z^3$, so in Figure \ref{fig: 3D}(a), we explore more points in the direction $z^1$ first, so $\t{N}(\{(0,0,0)\}) = \{(1,0,0)\}$. Then in Figure \ref{fig: 3D}(b), since we already have 2 points on the $z^1$-axis, instead of exploring more points on
the $z^1$ direction, we start to explore the $z^2$ direction, so $\t{N}(\{(0,0,0), (1,0,0)\}) = \{(2,0,0), (0,1,0)\}$. Assume after doing greedy search on $\t{N}(\L_1)$, one gets $\nu_2 = (0,1,0)$. Then in Figure \ref{fig: 3D}(c), Since one has two points on $z^1, z^2$ respectively, so one starts to explore more points on the third direction $z^3$ at this step. so $\t{N}(\{(0,0,0), (1,0,0),(0,1,0)\}) = \{(2,0,0),(0,2,0),(1,1,0),(0,0,1)\}$

Note that the size of $N^*(\L_n)$ depends on $j(\nu_n)$, and $j(\nu_n)\leq n$, so 
\begin{align}
	\#(\t{N}(\L_n)) \leq \#(\t{N}(\L_{n-1})) + n,
\end{align}
which gives the size of $\t{N}(\L_n)$ is at most 
\begin{align}
	\#(\t{N}(\L_n)) \leq \frac{1}{2}n(n-1)\sim O(n^2).
\end{align}

After one constructed the anchored neighbors of $\L_n$, one searches for the $\nu \in \t{N}(\L_n)$ that maximizes the interpolation error at the new grid point. In summary, we have the following algorithm.
\begin{algorithm}
\label{algo old}
(ASPI) \cite{chkifa2014high}
\begin{itemize}
\item [--] Step $0$. Construct the Leja sequence $\G = \{\b_j\}_{j\geq0}$ starting from $0$ as in (\ref{leja_const}) and the basis $\{l_k(\b)\}_{k\geq1}$ as in $(\ref{hierarchical basis})$.
\item [--] Step $1$. Define $\L_1 = \{0\}$ as the null multi-index and the corresponding polynomial interpolation $I_{\L_1}(\bz) = f(\bz_{\nu_1})$.
\item [--] Step $n$.\   Assume we already have $\t{N}(\L_{n-1})$, $\L_n$  and  $I_{\L_n}$.
\begin{itemize}
\item [-] Construct $\h{N}(\L_n)$ by (\ref{N_Ln}), then $N^*(\L_n)$ can be constructed through (\ref{def N*}), $\t{N}(\L_n)$ through (\ref{def of NLn}).
\item [-] $\displaystyle \nu_{n+1} = \argmax_{\nu\in\t{N}(\L_n)} \ll I_{\L_{n+1}} - I_{\L_n} \rl_{L^2(U,V,d\rho)} = \argmax_{\nu\in\t{N}(\L_n)}  \ll \a_{\nu}\rl_V \ll H_\nu\rl_{L^2(U,d\rho)} $,\\
where $\displaystyle \a_\nu(t,x,v)$, $H_\nu(\bz)$ are defined in (\ref{poly inter}), (\ref{hierarchical basis}) respectively.
\end{itemize}
\end{itemize}
\end{algorithm}

	After step $0$, the polynomial interpolation will be uniquely depending on the index set $\L$. In step $n$,  it gives the way to find the best $\L_n$ such that $I_{\L_n}(t, x, v, \bz)$ is  closest to $f(t, x,v,\bz)$ in the $L^2(U,V,d\rho)$ space. In the greedy search step, the reason why $\ll I_{\L_{n+1}} - I_{\L_n} \rl_{L^2(U,V,d\rho)} =  \ll \a_{\nu}\rl_V \ll H_\nu\rl_{L^2(U,d\rho)} $ is because of the progressive construction of the polynomial interpolation, which is interpreted in (\ref{poly inter}). $\argmax_{\nu\in \t{N}(\L_n)}$ is obtained by directly searching for the maximal value. 
	
	Furthermore, note that when calculating $\a_\nu$, one actually needs to calculate the function value $f(T,x,v,\bz_\nu)$ at point $\bz_\nu$. Although the final approximate solution $f(T,x,v,\bz)$ is a polynomial interpolating on  $N$ points $f(\bz_{\nu_1}), \cdots, f(\bz_{\nu_N})$, in order to get the best $\bz_{\nu_{n+1}}$ ($1\leq n \leq N$),  one needs to do the greedy search on $\t{N}(\L_{n})$, which includes calculating $f(\bz_\nu)$ for all $\nu\in\t{N}(\L_n)$. Since most PDEs have no analytic solution,  the computational cost of obtaining the solution at sample point $\bz_\nu$ highly depends on the numerical algorithm. We will see in the next section, the ASPI method is computationally inefficient for time dependent kinetic equation with small $\e$.

\subsection{The residual based adaptive sparse polynomial interpolation (RASPI)}
\label{RASPI}
As stated at the end of the previous section, we will explain in more details in this section why the ASPI is not as efficient as the RASPI for general linear kinetic equation.  The general form of a kinetic equation without uncertainty reads,
\begin{align}
	\pt_t f +v\pt_x f = \frac{1}{\e}Q(f),
\end{align}
where $f(t,x,v)$ is the probability density distribution of particles, $Q(f)$ describes the collision between particles. The parameter $\e$ represents the dimensionless mean free path or the Knudsen number,  which connects the microscopic kinetic model to the macroscopic hydrodynamic model when $\e\to0$. Kinetic equations give a uniform description of both mesoscopic and macroscopic physical quantities for all range of $\e$. A numerical scheme that preserves the asymptotic transitions from kinetic equations to their macroscopic limits in the numerically discrete space is called {\it{Asymptotic Preserving}} (AP) scheme \cite{Jin99, jin2010asymptotic}. For numerical stability independent of $\e$, the numerical scheme that is AP usually is implicit for the discretization of $Q(f)$. Let $f^m \in \R^M$ be the discretized vector for $f(m\d_t,x,v) $, where $\d_t$ is the time step, then the general form of the scheme for a linear $Q(f) = Bf$ with $B$ independent of $f$ is, 
\begin{align}
	\frac{f^{m+1} - f^m}{\d_t} + Af^{k} - \frac{1}{\e}B^{m+1}f^{m+1} = 0,
\end{align}
where $A, B^{m+1} \in \R^{M\times M}$ are constant matrices. 

For the kinetic equation with uncertainty in the collision operator, 
\begin{align}
	\pt_t f +v\pt_x f = \frac{1}{\e}Q(f,\bz),
\end{align}
for $\forall \bz\in U$, the general form of the scheme is 
\begin{align}
	\frac{f^{m+1}(\bz) - f^m(\bz)}{\d_t} + Af^m(\bz) - \frac{1}{\e}B^{m+1}(\bz)f^{m+1}(\bz) = 0.
	\label{implicit scheme}
\end{align}
For example, the VFP equation (\ref{VFP}) we considered in this paper, if one moves the forcing term $E\pt_vf$ to the RHS of the equation as is typically done in the high field regime \cite{jin2011asymptotic}, then the collision operator becomes, 
\begin{equation*}
    Q(f,z) = \pt_v\((v+E)f + \pt_vf\).
\end{equation*}
That's why in the most general case, the numerical operator $B$ depends on both $\bz$ and $t$. 
Equivalently, (\ref{implicit scheme}) can also be written as,
\begin{align}
	f^{m+1}(\bz) = \(\frac{1}{\d_t} - \frac{1}{\e}B^{m+1}(\bz)\)^{-1}\(\frac{f^m(\bz)}{\d_t} - Af^{k}(\bz)\).\label{inversion case}
\end{align}
This means that in order to calculate $f(T, x, v, \bz_\nu)$ at $T = N_t\d_t$, one needs to invert an $\R^{M\times M}$ matrix for $N_t$ times, where $M = N_x \times N_v$ with $N_x, N_v$ being the number of grid points in $x$ and $v$ respectively. So for each $\bz_\nu$, the cost is $O(N_tM^3)$. Algorithm \ref{algo old} requires calculating $f(T, x, v, \bz_\nu)$ for all $\bz_\nu \in \t{N}(\L_n)$, where the size of $\t{N}(\L_n)$ could be $O(n^2/2)$. So the ASPI method (Algorithm \ref{algo old}) for multi-scale kinetic equations could be computationally expensive, see Section \ref{compare} for the total cost.  

Next we will introduce an algorithm where calculating $f(T, x, v, \bz_\nu)$ for all $\bz_\nu \in \t{N}(\L_n)$ can be avoided.

At step $n$, we already have $\L_n = \{\nu_1, \cdots \nu_n\}$ and the numerical solution $f^m_{\nu_k} = f(m\d_t,\bz_{\nu_k})$ and $f^{m-1}_{\nu_k} = f((m-1)\d_t,\bz_{\nu_k})$ for $1\leq k \leq n$. Let operator $\mS$ be the numerical kinetic operator, 
\begin{align}
	\mS(f^m(\bz), f^{m-1}(\bz)) = \frac{f^m(\bz) - f^{m-1}(\bz)}{\d_t} + A^{m-1}f^{m-1}(\bz) - \frac{1}{\e}B^m(\bz)f^m(\bz).
\end{align}
For $f^m_{\nu_k}, f^{m-1}_{\nu_k}$ obtained from the numerical scheme, it must satisfy $\mS(f^m_{\nu_k}, f^{m-1}_{\nu_k}) = 0$. For the interpolation approximations $I^m_{\L_n}(\bz), I^{m-1}_{\L_n}(\bz)$ interpolating on $f^m_{\nu_k}, f^{m-1}_{\nu_k}$ respectively, $1\leq k\leq n$,  $\mS(I^m_{\L_n}(\bz), I^{m-1}_{\L_n}(\bz))$ represents the residual of the scheme for the polynomial interpolation $I^m_{\L_n}$ at $\bz$. So one can search for the biggest residual with respect to $\mS(I^m_{\L_n}(\bz), I^{m-1}_{\L_n}(\bz))$ on $\nu \in \t{N}(\L_n)$ to get $\nu_{n+1}$. We will later see that the greedy search in this way costs less than the ASPI method. Since the interpolation $I^m_{\L_n}(\bz)$ on data $f^m_{\nu_k}$  for $1\leq k \leq n$ can be represented by a linear combination of $f^m_{\nu_k}$, which can be written as
\begin{align}
\label{b_k}
	I^m_{\L_n}(\bz) = \sum_{k = 1}^n\g_{\L_n}^k(\bz)f^m_{\nu_k},
\end{align}
where 
\begin{align}
	\begin{bmatrix}  
		\g^1_{\L_n}(\bz)
		\cdots
		\g^n_{\L_n}(\bz)
	\end{bmatrix} = \begin{bmatrix}  
		H_{\nu_1}(\bz)
		\cdots
		H_{\nu_n}(\bz)
	\end{bmatrix}
	\begin{bmatrix}  
		H_{\nu_1}(\bz_{\nu_1})\cdots H_{\nu_n}(\bz_{\nu_1})\\
		\vdots\\
		H_{\nu_1}(\bz_{\nu_n})\cdots H_{\nu_n}(\bz_{\nu_n})\\
	\end{bmatrix}^{-1},  \label{g_k}
\end{align}
hence, plugging (\ref{b_k}) into operator $\mS$ gives
\begin{align}
	\mS(I^m_{\L_n}(\bz), I^{m-1}_{\L_n}(\bz)) =&\sum_{k=1}^n \g^k_{\L_n}\(\frac{f^m_{\nu_k} - f^{m-1}_{\nu_k}}{\d_t} + A^{m-1}f^{m-1}_{\nu_k}\) -\frac{1}{\e}B^m(\bz) \sum_{j=1}^n \g^k_{\L_n}f^m_{\nu_k}\nonumber\\
	=&\sum_{k=1}^n\g^k_{\L_n}\(\frac{1}{\e}B^m_{\nu_k} f^m_{\nu_k}\) -\frac{1}{\e}B^m(\bz) \sum_{k=1}^n \g^k_{\L_n}f^{m}(\bz_{\nu_k})\nonumber\\
	=&\frac{1}{\e} \sum_{k=1}^n \g^k_{\L_n}\(B^m_{\nu_k}f^m_{\nu_k} - B^m(\bz)f^m_{\nu_k}\) , 
\end{align}
where the second equality is because of $S(f^m_{\nu_k}, f^{m-1}_{\nu_k}) = 0$. 

In addition, when calculating $\g_{\L_n}(\bz) = [\g_1(\bz), \cdots, \g_n(\bz)]$ in (\ref{g_k}), we don't need to invert the whole matrix on the RHS at every step. Because of the monotonicity of $\L_n = \L_{n-1} \cup \{\nu_n\}$, and the Schur complement of the inversion from the previous step,  we can avoid computing the inversion. Specifically, define
\begin{align}
	H_{\L_n}(\bz_{\L_n}) : = 
	\begin{bmatrix}  
		H_{\nu_1}(\bz_{\nu_1})\cdots H_{\nu_n}(\bz_{\nu_1})\\
		\vdots\\
		H_{\nu_1}(\bz_{\nu_n})\cdots H_{\nu_n}(\bz_{\nu_n})\\
	\end{bmatrix}, \qd H_{\L_{n-1}}(\bz_\nu) = [H_{\nu_1}(\bz_\nu),\cdots, H_{\nu_{n-1}}(\bz_\nu)],
\end{align}
then by (\ref{H_nu}), $H_{\L_n}(\bz_{\L_n})$ can also be written in the form of block matrix,
\begin{align}
	H_{\L_n}(\bz_{\L_n}) = 
	\begin{bmatrix}  
		&H_{\L_{n-1}}(\bz_{\L_{n-1}}) & {\bf{0}}_{ (n-1)\times1}\\
		&H_{\L_{n-1}}(\bz_{\nu_n})  &1
	\end{bmatrix}.  
\end{align}
It is easy to check that, 
\begin{align}
	H_{\L_n}^{-1}(\bz_{\L_n})= 
	\begin{bmatrix}  
		&H_{\L_{n-1}}^{-1}(\bz_{\L_{n-1}})   &0\\
		&-\g_{\L_{n-1}}(\bz_{\nu_n})  &1
	\end{bmatrix}.  \label{H_inv}
\end{align}
Let 
\begin{align}
	S^{\L_n}_\nu = \sum_{k=1}^n \g_{\L_n}^k(\bz_\nu)\(B_{\nu_k}f_{\nu_k} - B_\nu f_{\nu_k}\) \in\R^M \label{residual calc}
\end{align} be the residual of interpolation of $\bz_\nu$ at time $T$, where $f_{\nu_k} = f(T, \bz_{\nu_k})$, $B_\nu = B(T, \bz_\nu)$, so based on this residual, we construct the following new algorithm.

\begin{algorithm}
\label{algo new}
(RASPI)
\begin{itemize}
\item [--] Step $0$, Step $1$ are the same as Algorithm \ref{algo old}.
\item [--] Step $n$.\   Assume we have $\L_n$, $\t{N}(\L_{n-1})$ and  $I_{\L_n}(\bz)$, 
\begin{itemize}
\item [-] Construct $\h{N}(\L_n)$ by (\ref{N_Ln}), then $N^*(\L_n)$ can be constructed through (\ref{def N*}), $\t{N}(\L_n)$ through (\ref{def of NLn}).
\item [-] $\nu_{n+1} = \argmax_{\nu\in\t{N}(\L_n)} \ll S^{\L_n}_\nu \rl_V$, where $S^{\L_n}_\nu$ is defined in (\ref{residual calc}).
\end{itemize}
\end{itemize}
\end{algorithm}

Compared with Algorithm \ref{algo old}, the above algorithm is more efficient since for each $\bz\in \t{N}(\L_n)$, one only needs to  multiply an $\R^{M\times M}$ matrix to a $M-$dimensional vector once. The computational cost for each $\bz_\nu$ is $O(M^2)$, which is much less than $O(N_tM^3)$.  We will compare the total computational cost of the two algorithms in details in the next section.

\subsection{Computational cost}
\label{compare}
In this section, we will compare the computational costs between the ASPI (Algorithm \ref{algo old}) and the RASPI (Algorithm \ref{algo new}). In order to get an approximate solution with an error less than $\d$, for a first order discretization in the phase space, one needs to use  $N_x = O(\d^{-1})$, $N_v = O(\d^{-1})$. So  the explicit expression of the approximate solution $\hf(T,\bz) = I_{\L_N}(\bz)$ at time $T$  should be an $M$-dimensional vector where $M = N_x\cdot N_v= O(\d^{-2})$. In the random space, according to Theorem \ref{conv rate of t}, the best $N$ approximation gives the error $\ll f - \hat{f}\rl_{L^2(V,U,d\rho)} \leq N^{-s}$, thus one requires $N = O\(\d^{-1/s}\)$ to get an $O(\d)$ error.

For the ASPI method, at the $n$-th step of Algorithm \ref{algo old}, one needs to do the following calculation:
\begin{enumerate}
\item [\circled{0}] From the previous steps, one has,
\begin{itemize}
\item [-]  $\a_{\nu_k}\in\R^M$,  for all $\nu_k\in\L_n$;  
\item [-]  $\a_\nu\in\R^M$,  for all $\nu\in\t{N}(\L_{n-1})$.
\item [-]  $c_\nu= \ll \a_\nu \rl \ll H_\nu \rl_{L^2(V,d\rho)}$,  for all $\nu\in\t{N}(\L_{n-1})$.
\end{itemize}
\item [\circled{1}] Obtain $f(T,\bz_\nu)\in\R^M$ by numerical scheme (\ref{implicit scheme}), for all $\nu\in N^*(\L_n)$. 
\item [\circled{2}] Obtain $\a_\nu = f(T,\bz_\nu) - I_{\L_n}(\bz_\nu)$ for all $\nu \in N^*(\L_n)$.
\begin{itemize}
\item [-] To get the value of $I_{\L_n}(\bz_\nu)$, one needs to do the summation $I_{\L_n}(\bz_\nu) = \sum_{k = 1}^n \a_{\nu_k}H_{\nu_k}(\bz_\nu)$.
\end{itemize}
\item  [\circled{3}]  Obtain $c_\nu = \ll \a_\nu \rl \ll H_\nu \rl_{L^2(V,d\rho)}$ for all $\nu\in N^*(\L_n)$ and find $\nu_{n+1} = \argmin_{\nu\in \t{N}(\L_{n})} c_\nu$.
\end{enumerate}

In step  \circled{1}, one needs to calculate the numerical solution to the PDE at time $T$ for all $\bz_\nu\in N^*(\L_n)$,  where the size of $N^*(\L_n)$ is  $O(n)$. For general implicit scheme as (\ref{implicit scheme}), the computational cost to obtain $f(T,\bz)\in\R^M$ is $O((M^3+M^2)N_t)$, where $M^3$ comes from the inversion of matrix, $M^2$ comes from the multiplication of matrices, and these have to be done in each step.   Therefore, the computational cost in step \circled{1} is
\begin{align}
	O(n\times(M^3+M^2)N_t ).
\end{align}
There are also cases where the inversion can be completed within a cost of $O(M^2)$, or the inversion only needs to be done once if $B$ is time independent, then the computational cost of these are calculated in Remark \ref{rmk: other cases}.

In step \circled{2}, for each $\nu\in N^*(\L_n)$, the computational cost to get $I_{\L_n}(\bz_\nu)$ is $O\( nM\r.$\\ $\l.+\sum_{k=1}^n\text{Cost}\{ H_{\nu_k}(\bz_\nu)\} \)$, where the cost of $H_{\nu_k}(\bz_\nu)$ is $O(k)$ for each $\nu$. Hence one requires
\begin{align}
	O\(n \times \( nM + n^2\)\)
\end{align}
of computational operations to complete step \circled{2}. 

At last, calculating $\ll \a_\nu \rl \ll H_\nu\rl$ for each $\nu\in N^*(\L_n)$ requires $O(M + n\sqrt{M})$ operations. Then searching for the smallest one requires $O\(\#\(\t{N}(\L_n)\)\)$ operations. Hence  the total  cost is
\begin{align}
	O\( n \times \(M+n\sqrt{M}\) + n^2\).
\end{align}

To sum up, the total cost at the $n$-th step of the ASPI method is 
\begin{align}
	O\(nM^3N_t + n^2M + n^3\) .
\end{align}
Plug in $M = O(\d^{-2})$, and assume time $T\sim O(1)$ and $\d_t \sim \d_x$, so $N_t \sim O(\d^{-1})$, hence the total computational cost at the n-th step of Algorithm \ref{algo old} is 
\begin{align}
	O\(\d^{-7}n+ \d^{-2}n^2  + n^3 \). \label{cost of algo old}
\end{align}

While for the RASPI method, one needs to do the following calculation at the $n$-th step,
\begin{enumerate}
\item [\circled{0}] From the previous steps, one has,
\begin{itemize}
\item [-] $H^{-1}_{\L_{n-1}}\(\bz_{\L_{n-1}}\)$;
\item [-]  ${\bm{\g}}_{\L_{n-1}}(\bz_\nu)$, for $\nu\in\t{N}(\L_{n-1})$ and $\nu = \nu_n$;
\item [-]  $f_{\nu_k},B_{\nu_k}f_{\nu_k}$, $B_\nu f_{\nu_k}$ for all $\nu_k\in\L_n$, $\nu\in \t{N}(\L_{n-1})$ .
\end{itemize}
\item [\circled{1}] Obtain $\bg_n(\bz_\nu) $, for all $\nu\in \t{N}(\L_n)$:
\begin{itemize}
\item [-] Get $H^{-1}_{\L_n}$ by (\ref{H_inv}).
\item [-] For $\nu\in \t{N}(\L_{n-1})$, $\bg_{\L_n}(\bz_\nu) =[ \bg_{\L_{n-1}}(\bz_\nu) - H_{\nu_n}(\bz_\nu)\bg_{\L_{n-1}}(\bz_{\nu_n}), H_{\nu_n}(\bz_\nu)]$.
\item [-] For $\nu\in N^*(\L_n)$, $\bg_n(\bz_\nu) = [H_{\L_{n-1}}(\bz_\nu) H^{-1}_{\L_{n-1}}- H_{\nu_n}(\bz_\nu)\bg_{\L_{n-1}}(\bz_{\nu_n}), H_{\nu_n}(\bz_\nu)]$.
\end{itemize}
\item [\circled{2}] Obtain $S^{\L_n}_\nu$ by (\ref{residual calc}).
\item  [\circled{3}]  Find $\nu_{n+1} = \argmin_{\nu\in \t{N}(\L_{n})} \ll S^{\L_n}_\nu\rl^2$.
\end{enumerate}

Firstly in step \circled{1}, since one already has $H^{-1}_{\L_{n-1}}\(\bz_{\L_{n-1}}\)$ and ${\bm{\g}}_{\L_{n-1}}(\bz_\nu)$ from the previous step, one only needs to plug them in to get $H^{-1}_{\L_n}$.  For each $\nu\in\t{N}({\L_{n-1}})$, one needs $O(n)$ operations to get $\bg_n(\bz_\nu)$. While for each $\nu\in N^*(\L_n)$, one needs $O((n-1)^2+ \sum_{k=1}^n\text{Cost}\{ H_{\nu_k}(\bz_\nu)\})$ operations to get $\bg_n(\bz_\nu)$.  The total computational cost is
\begin{align}
	O\((n-1)^2\times n\)+ O\(n \times \((n-1)^2+n^2\)\).
\end{align}

In step \circled{2},  for each $\nu\in\t{N}(\L_{n-1})$, since one already has $B_{\nu_k}f_{\nu_k} - B_\nu f_{\nu_k}$ from the previous step, so one only needs to do the weighted sum operations given $\bg_{\L_n}(\bz_\nu)$ , which requires $O(nM)$ computational cost.  For each $\nu\in N^*(\L_n)$, one needs to calculate $B_\nu f_{\nu_k}$ first then does the summation, whose computational cost is $O(M^2+nM)$. Therefore the total computational cost in \circled{2} is
\begin{align}
	O((n-1)^2 \times nM)+ O(n\times(M^2+nM)).
\end{align}

At last, Obtaining $\ll S(T,\bz_\nu) \rl^2$ and finding the minimum among all $\nu\in \t{N}(\L_n)$ requires computational cost of order 
\begin{align}
	O(n^2 \times M  + n^2).
\end{align}

To sum up, the total cost at $n$-th step of the RASPI is 
\begin{align}
	O\(n^3M + nM^2\).
\end{align}
Plugging in $M = O(\d^{-2})$ gives the total cost of Algorithm \ref{algo new} at the $n$-th step
\begin{align}
	O(n\d^{-4} + n^3\d^{-2}). 
	\label{cost of algo new}
\end{align}

Summing (\ref{cost of algo old}) and (\ref{cost of algo new}) over $1\leq n\leq O\(\e^{-1/s}\)$, and based on the fact that $1/s \leq 2$, one has
\begin{align}
	&\text{Computational cost of ASPI}: O(\d^{-7-2/s}),\\
	&\text{Computational cost of RASPI}: 
	\begin{cases}
	O(\d^{-2-4/s}), \qd \frac{1}{2}\leq s\leq 1\\
	O(\d^{-4-2/s}), \qd s \geq1
	 \end{cases}
\end{align}
The ratio of the two costs is
\begin{align}
	\frac{\text{Computational cost of ASPI}}{\text{Computational cost of RASPI}}=
	\begin{cases}
	O(\d^{-5+2/s}), \qd \frac{1}{2}\leq s\leq 1,\\
	O(\d^{-3}), \qd \qd \qd s \geq1.
	 \end{cases}
	 \label{eq: compare}
\end{align}
From (\ref{eq: compare}), one can see that the computational cost of the ASPI is $O(\d^{-5+2/s})$ times that of RASPI for $s\leq 1$ and $O(\d^{-3})$ times that of the  RASPI for $s\geq1$. Since $s\geq \frac{1}{2}$, therefore the RASPI is always more efficient. In addition, the faster $\ll \p_j(x) \rl_{W^{1,\infty}_x}$ decays, the more computational cost the RASPI saves. 

\begin{remark}
\label{rmk: other cases}
\begin{enumerate}
\item The inversion of a matrix in (\ref{inversion case}) does not necessarily need the cost of $O(M^3)$. For example, when the matrix is positive definite, one can invert  an $\R^{M\times M}$ matrix by the conjugate gradient method with computational cost of $O(M^2)$.  Also, when the collision operator $Q(f)$ is time independent, then the matrix $B^{m+1}$ in the numerical method (\ref{inversion case}) is the same constant matrix for all $ m$, so one only needs to invert the matrix once. In both cases, the computational cost of calculating $f(T,x,v,\bz_\nu)$ for a specific $\bz_\nu\in U$ is $O(M^2N_t)$, which reduces the total computational cost of the ASPI to $O(\d^{-5-2/s})$. Then the ratio of the two costs becomes
\begin{align}
	\frac{\text{Computational cost of ASPI}}{\text{Computational cost of RASPI}}=
	\begin{cases}
	O(\d^{3-2/s}), \qd \frac{1}{2}\leq s< \frac{2}{3},\\
	O(\d^{-(3-2/s)}), \qd \frac{2}{3}\leq s\leq 1,\\
	O(\d^{-1}), \qd \qd \qd s \geq1.
	 \end{cases}
\end{align}
When $\frac{1}{2}\leq s<\frac{2}{3}$, the ASPI method is more efficient than the RASPI, while $s\geq\frac{2}{3}$, RASPI is still better than ASPI. 

\item Another deterministic method called Quasi Monte Carlo (QMC) is also widely used in parametric PDEs. However, in general, since the convergence rate of QMC is $O(\frac{log(N)^d}{N})$ \cite{morokoff1995quasi, caflisch1998monte}, which depends on dimensionality of the parameter, so it is not comparable in high dimension. Nevertheless, as discussed  in \cite{kuo2015multi, kuo2012quasi} for parametric elliptic equation using modified QMC method, it enjoys the same convergence rate as the best N approximation when the randomness $\ll \psi_j \rl \in \ell^p$ for $ 2/3 \leq p \leq 1$. Its performance for kinetic equations remain to be investigated.

\item  In general whether the ASPI method can achieve the error estimates we get in Section \ref{theoretical} is still an open question. However, under stronger assumptions, \cite{ZDS18_759} showed that a certain type of adaptive 
sparse grid interpolation will produce sequences of active index sets in polynomial basis function space 
which will give a dimension independent convergence rate.
\end{enumerate}
\end{remark}

\section{Numerical examples}
\label{numerical eq}
In this section, we conduct some numerical experiments for the linear Vlasov-Fokker-Planck equation with random electric field $E(t,x,\bz)$ , 
\begin{align}
\e\pt_t f + v\pt_xf - E\pt_v f = \frac{1}{\e}\mF f, \qd x,v \in \O = (0,2\pi) \times \R \label{eqn_eq1}
\end{align}
with periodic condition on $x\in [0,2\pi]$, and initial data 
\begin{align}
	f(0,x,v) = \frac{\sin(x)}{\sqrt{2\pi}}e^{-\frac{v^2}{2}} + F,\label{ID_eq1}
\end{align}
where $F = e^{-\pinf}M$ is defined in (\ref{def of F}). We consider $\bz\in[-1,1]^{100}$, and set the electric field as,
\begin{align}
	E(t,x,\bz) = \frac{\sin(x)}{2} + \sum_{j=1}^{100}E_j(t,x)z_j \label{form}
\end{align}
with different choices of $E_j(t, x)$ in the experiments. We solve (\ref{eqn_eq1})  by finite difference method with unified meshes in space and velocity $\d_x = \frac{2\pi}{Nx}$ on $[0,2\pi]$ and $\d_v = \frac{12}{N_v}$ on $[-6,6]$. The scheme we use here is from \cite{jin2011class, jin2011asymptotic}
\begin{align}
	\e\frac{f^{m+1}_{i,j} - f^{m}_{i,j}}{\d_t} + \frac{f^{m+1}_{i+\frac{1}{2},j} - f^{m+1}_{i-\frac{1}{2},j}}{\d_x} = \frac{1}{\e}P_{\bz}\(\frac{f^{m+1}_{i,j}}{\sqrt{M_{i,j}}}\),
	\label{scheme}
\end{align}
where $f^m_{i,j} = f(m\d_t,i\d_x,j\d_v-6)$. The transport term $v\pt_xf$
is approximated by the upwind scheme
$$f^m_{i+\frac{1}{2},j} = \frac{|v_j|+v_j}{2}f^m_{i,j} - \frac{|v_j|-v_j}{2}f^m_{i+1,j}.$$ For the other terms,  since 
$$\frac{1}{\e} \(\e E\pt_vf+ \mF(f)\) = \frac{1}{\e} \pt_v\(M_l\pt_v\(\frac{f}{M_l}\)\)$$ 
with 
$$M_l(\bz) = \frac{1}{\sqrt{2\pi}}e^{-\frac{|v-\e E|^2}{2}},$$
depending on $\bz$, so one can define an operator $P_{\bz}(f)$ as the discretization of $\pt_v\(M_l\pt_v\(\frac{f}{M_l}\)\)$ as following, 
\begin{align}
	P_\bz(f_j) =&  \frac{1}{\d_v^2}\(\sqrt{(M_l)_{j+1}(M_l)_j}\(\frac{f_{j+1}}{(M_l)_{j+1}} - \frac{f_{j}}{(M_l)_{j}}\) - \sqrt{(M_l)_{j-1}(M_l)_j}\(\frac{f_{j}}{(M_l)_{j}} - \frac{f_{j-1}}{(M_l)_{j-1}}\) \).\label{def of P_z}
\end{align}
Since the scheme is in implicit form, for the operator above, one needs to do $N_xN_t$ times inversion of a $\R^{N_v\times N_v}$ matrix for each numerical solution $f(T,x,v,\bz)$ at a sample point $\bz$. One efficient way to reduce the computational cost is to set \cite{jin2011class}
\begin{align}
	g^m_{i,j} = \frac{f^m_{i,j}}{\sqrt{(M_l)^m_{i,j}}},
\end{align}	
 then 
\begin{align}
	\t{P}_\bz(g_j) =&  \frac{\sqrt{(M_l)_{j}}}{\d_v^2\sqrt{(M_l)_j}}\(\sqrt{(M_l)_{j+1}(M_l)_j}\(\frac{g_{j+1}}{\sqrt{(M_l)_{j+1}}} - \frac{g_j}{\sqrt{(M_l)_j}}\) \r.\nonumber\\
	&\l.- \sqrt{(M_l)_{j-1}(M_l)_j}\(\frac{g_j}{\sqrt{(M_l)_j}} - \frac{g_{j-1}}{\sqrt{(M_l)_{j-1}}}\) \)\nonumber\\
	=&  \frac{\sqrt{(M_l)_{j}}}{\d_v^2}\(g_{j+1} -\( \frac{\sqrt{(M_l)_{j+1}}}{\sqrt{(M_l)_j}} +  \frac{\sqrt{(M_l)_{j-1}}}{\sqrt{(M_l)_j}}\)g_j + g_{j-1} \). 
\end{align}
In this way, scheme  (\ref{scheme}) becomes, 
\begin{align}
	\e\frac{g^{m+1}_{i,j} - f^{m}_{i,j}/\sqrt{(M_l)^{m+1}_{i,j}}}{\d_t} + \frac{g^{m+1}_{i+\frac{1}{2},j} - g^{m+1}_{i-\frac{1}{2},j}}{\d_x} = \frac{1}{\e}\t{P}_{\bz}\(g^{m+1}_{i,j}\), 
	\label{new scheme}
\end{align}
therefore we can get a symmetric positive definite matrix multiplied to $g_i^{m+1}$, which can be inverted with less computational cost, for example, by the conjugate gradient method. 

For all numerical experiments, we set $N_x = 32, N_v = 64, \d_t = \frac{\d_x}{8}, T = 0.1, \e = 1$. 

\subsection{Convergence rate}
We test three different time independent random electric fields in the form of (\ref{form}), where $E_j(x)$ is given by the following functions:
\begin{align}
	&a) \ E_j(x) =\frac{\cos(jx)}{2^j};
	\qd b) \  E_j(x) = \frac{\cos(jx)}{j^2};
	\qd c) \ E_j(x) =\frac{\cos(jx)}{j}.\label{eq_E}
\end{align}    

Let $f(T,\bz)$ represent the numerical solution obtained from scheme (\ref{scheme}),  $\hf(T,\bz)$ represents the approximate solution obtained by sparse polynomial interpolation  on sample points $\G_{\L_n}$. Specifically, for the ASPI algorithm, we get $\a_{\nu_n}$, $1\leq n \leq N$, then $\hf(\bz_i) = \sum_{n = 1}^N\a_{\nu_n}H_{\nu_n}(\bz_i)$. Similarly, for the RASPI algorithm, we get $H^{-1}_{\L_N}(\bz_{\L_N})$ and $f_{\L_N}$, then $$[\a_{\nu_1}, \cdots, \a_{\nu_N}] = H^{-1}_{\L_N}(\bz_{\L_N}) f_{\L_N},$$ hence one can get $\hf(T,\bz_i) = \sum_{n = 1}^N\a_{\nu_n}H_{\nu_n}(\bz_i)$.

For the convergence rate, we check the mean square error defined as following, 
\begin{align}
\label{def of error}
	\text{Error} = \sqrt{\frac{1}{N}\sum_{i = 1}^{N} \ll \hf(T,\bz_i) - f(T,\bz_i) \rl^2_{L^2_{x,v}}}
\end{align}
with $N = 10^5$, where $\bz_i$ is uniformly drawn from $[-1,1]^{100}$, to test the accuracy of the sparse interpolation.


The left column of Figure \ref{fig_1} shows how the error decays when adding sample points adaptively by the ASPI method  and the RASPI method. From the numerical results one can see that both methods enjoy almost the same convergence rate.  The convergence rates are different for three different electric fields.  By comparing the decay rate of error for each example, one finds that if $\ll E_j(x)\rl_V$ decays faster, then the approximation also converges with a faster rate. We further show the algebraic decay rate $s$ in the slope.  For  $ E_j(x)$ given in (\ref{eq_E}) that decays in the order of $O\(2^{-j}\)$, $O\(j^{-2}\)$, $O\(j^{-1}\)$ respectively, the decay rate of the error in terms of the number of basis or number of sampling points is about $O\(N^{-2,7}\)$, $O\(N^{-1.2}\)$, $O\(N^{-0.7}\)$ respectively for $N$ basis or sample points, which are all faster than the Monte Carlo method of $O\(N^{-0.5}\)$.
\begin{figure}[htbp]
\subfloat[The convergence rate for the case of $\displaystyle E(x,\bz) = \frac{\sin(x)}{2} + \sum_{j=1}^{100}\frac{\cos(jx)}{2^j}z_j$]{\includegraphics[width=0.5\textwidth]{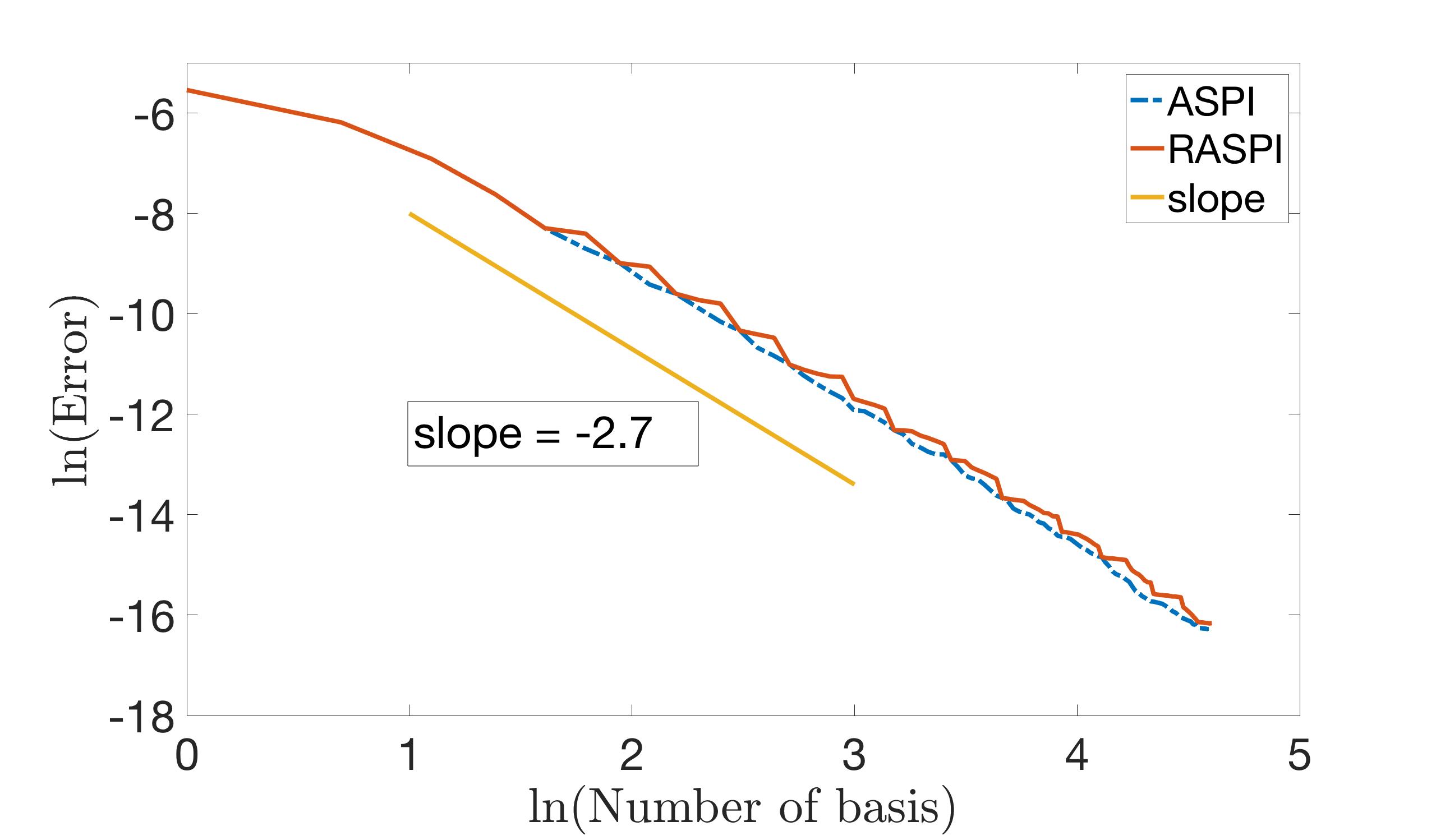}\qd\includegraphics[width=0.5\textwidth]{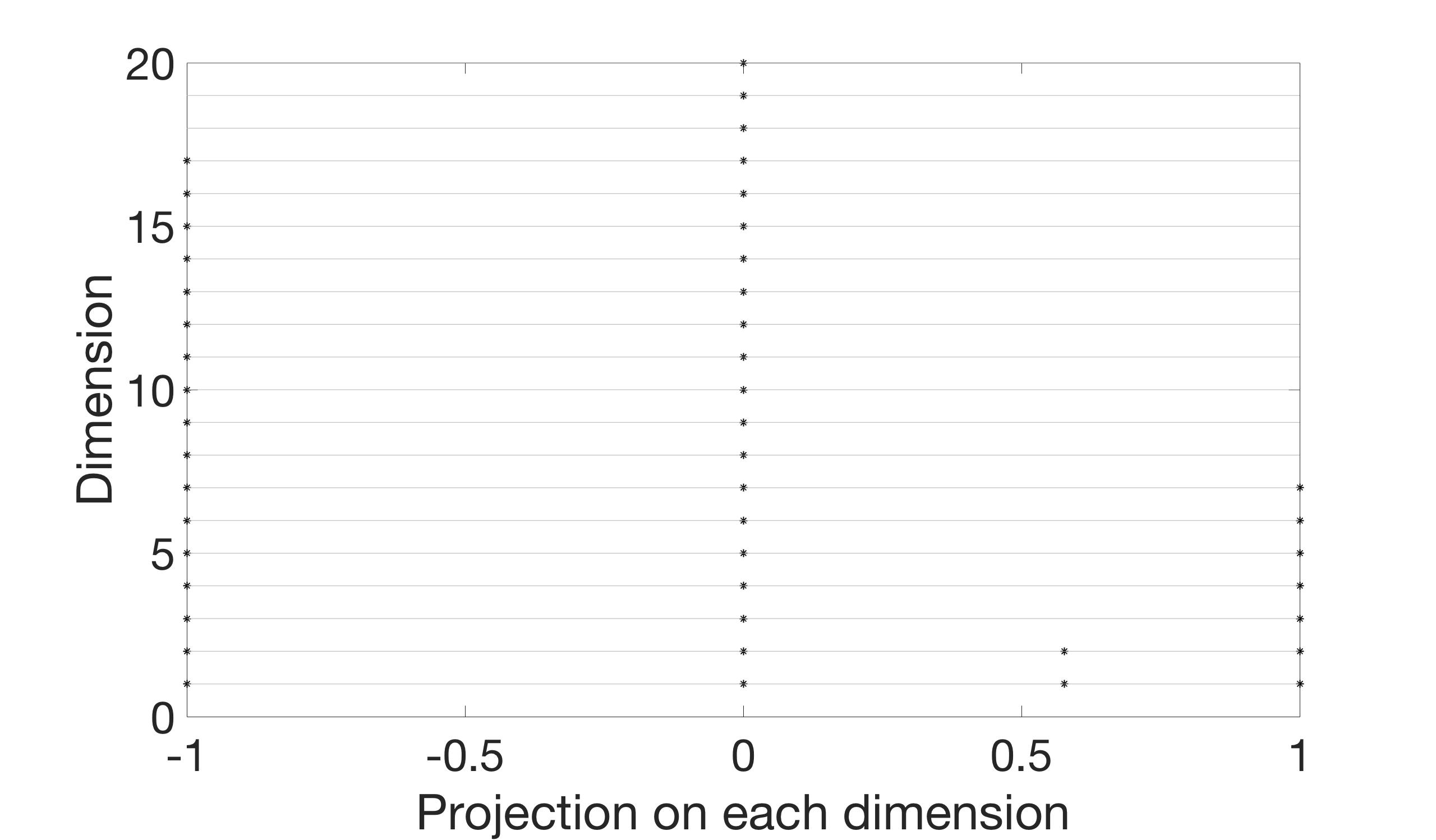}}\\
\subfloat[The convergence rate for the case of $\displaystyle E(x,\bz) = \frac{\sin(x)}{2} + \sum_{j=1}^{100}\frac{\cos(jx)}{j^2}z_j$]{\includegraphics[width=0.5\textwidth]{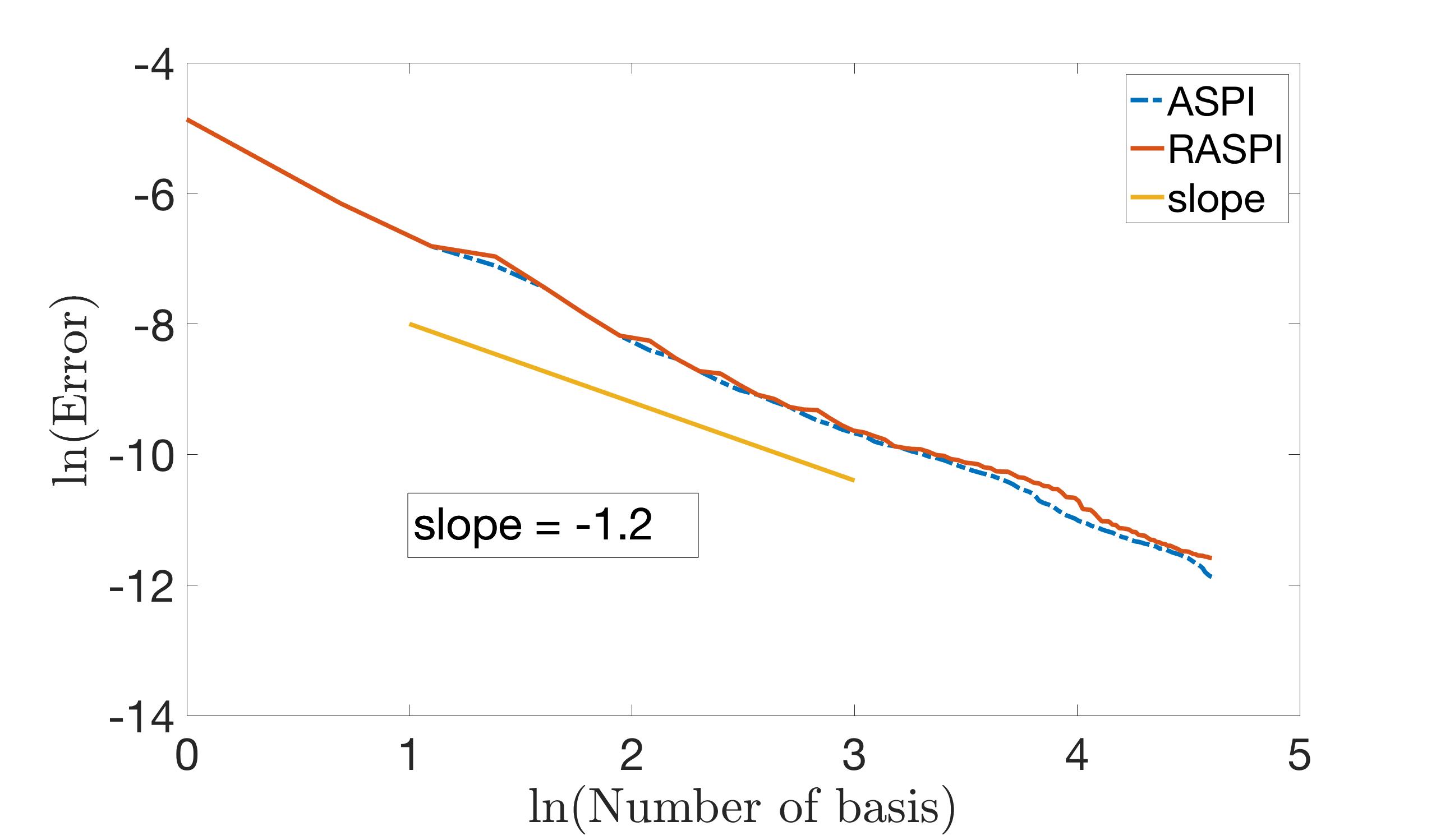}\qd\includegraphics[width=0.5\textwidth]{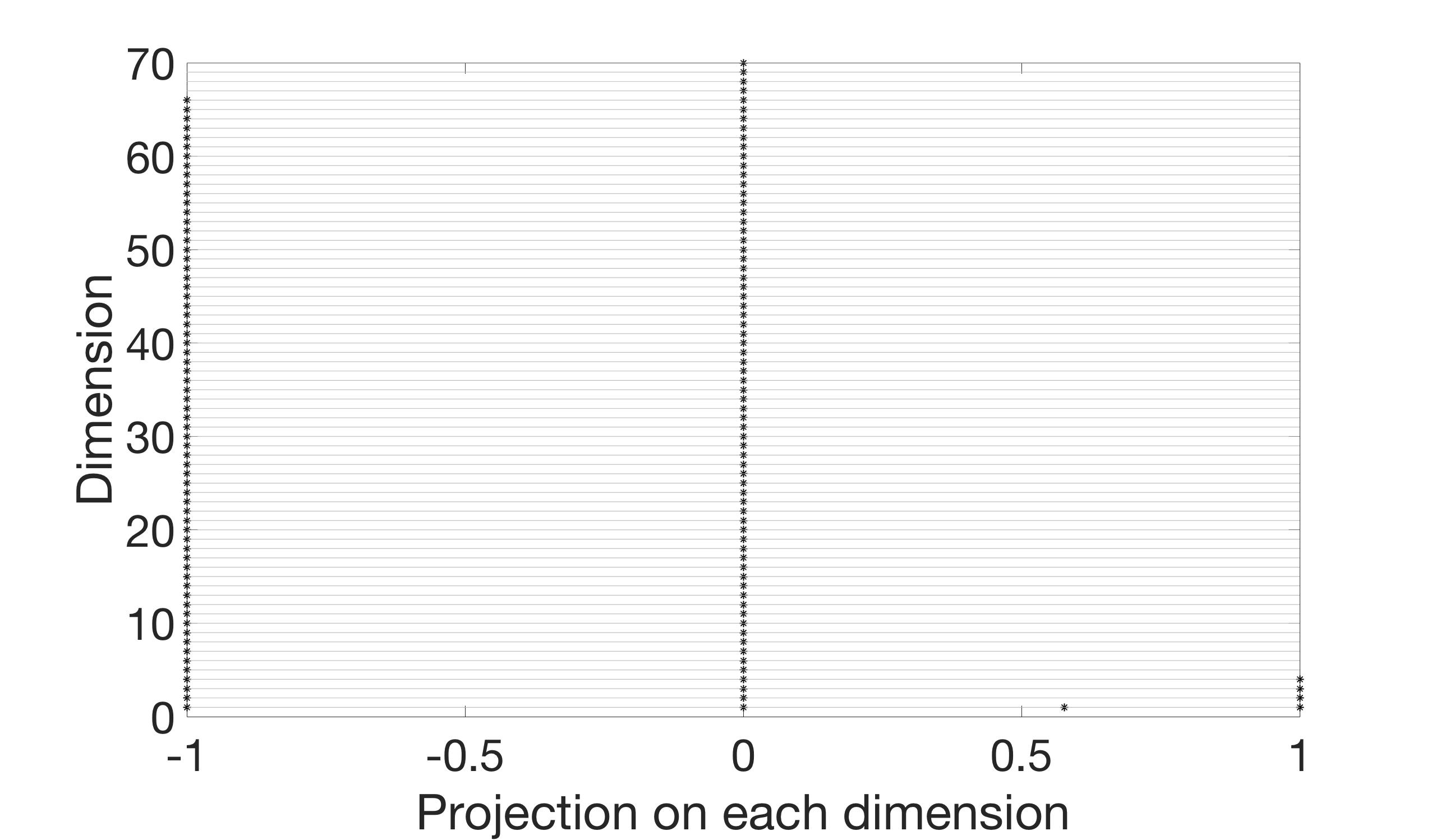}}\\
\subfloat[The convergence rate for the case of $\displaystyle E(x,\bz) = \frac{\sin(x)}{2} + \sum_{j=1}^{100}\frac{\cos(jx)}{j}z_j$]{\includegraphics[width=0.5\textwidth]{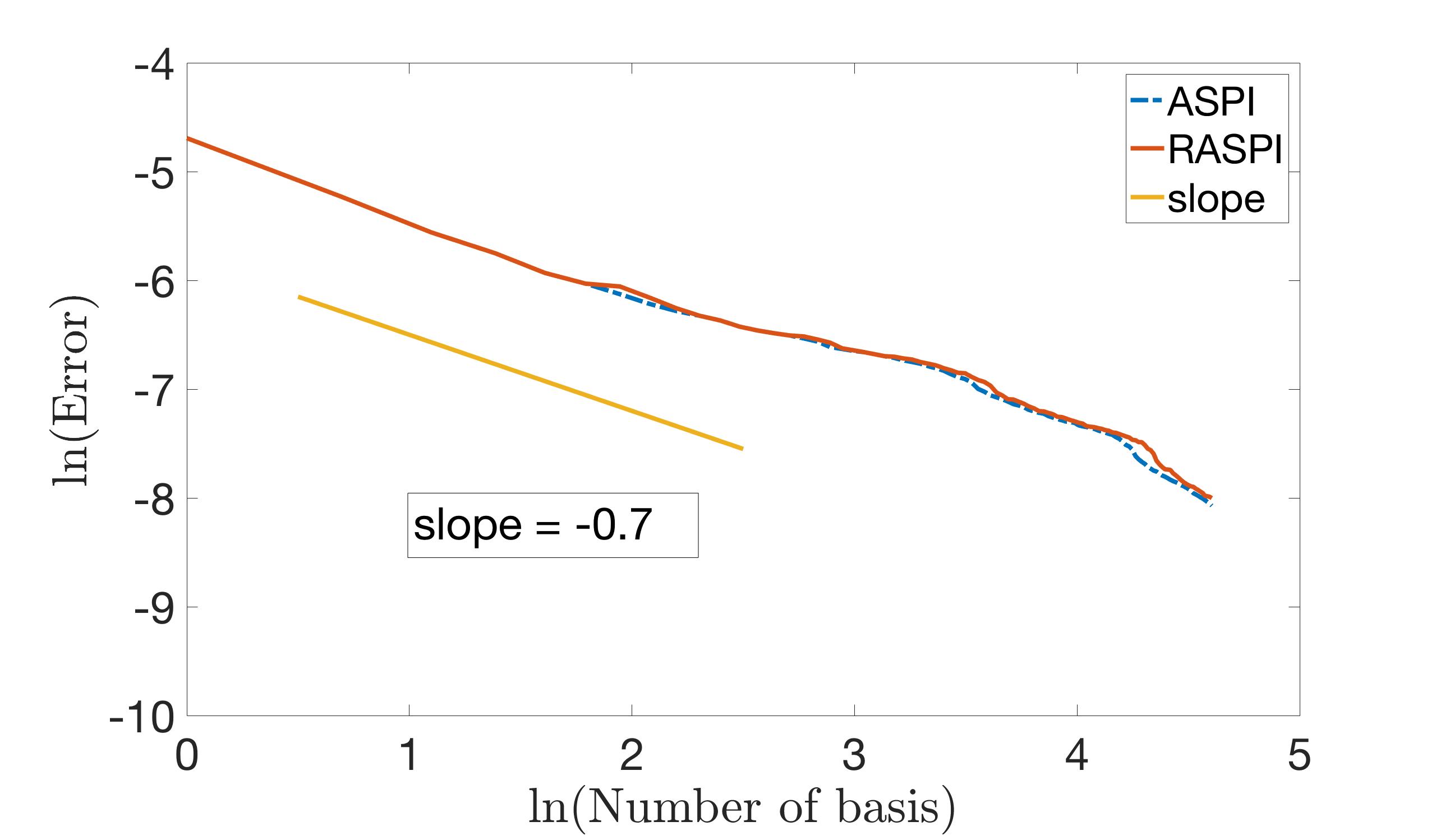}\qd\includegraphics[width=0.5\textwidth]{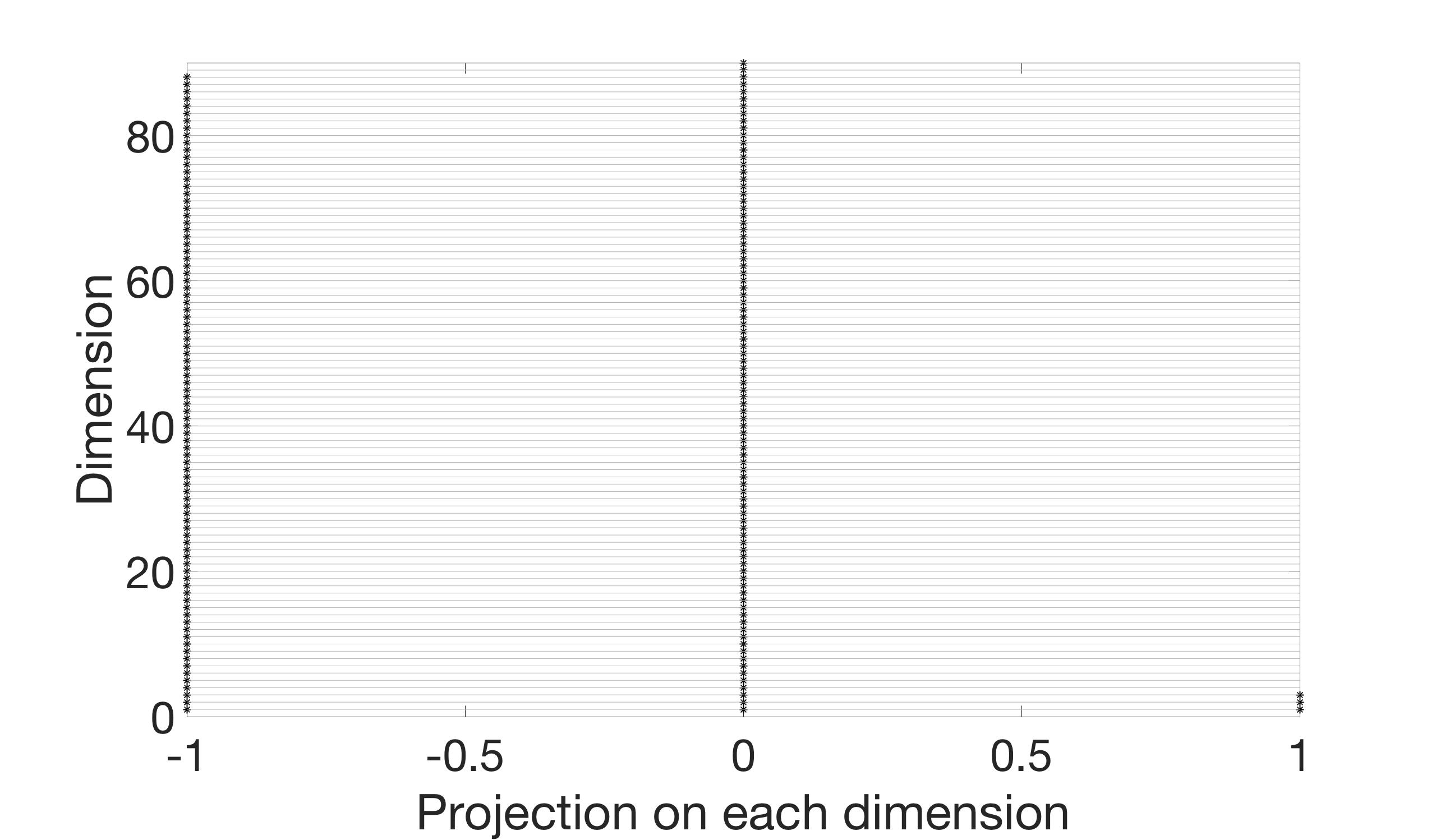}}
\caption{The convergence rate for error defined in (\ref{def of error}) of the approximate solution to (\ref{eqn_eq1}) at $t=1$ with different $E$ obtained by the ASPI method and the RASPI method, for $\e = 1$.}
\label{fig_1}
\end{figure}

However, as stated at the end of Section \ref{ASPI}, for the ASPI method, in order to get the optimal $N$ bases, one needs to compare all possible $a_{\nu}$ for $\nu\in \t{N}(\L_n)$, which involves the computation of the solution $f(T,x,v,\bz)$ at $z_\nu$ for $\nu\in \t{N}(\L_n)$. In other words, the number of  sample points used is $\#(N(\L_{100}))$, which is much larger than $100$. For example, it is 300, 2325, 3933 for case (a), (b), (c) respectively. By taking all of these sample points into account, the decay rate of ASPI corresponding to the number of sample points are shown in Figure \ref{fig_2}. For each sample the decay rate for each example is $O\(N^{-2}\)$, $O\(N^{-0.9}\)$, $O\(N^{-0.3}\)$ respectively. In particular, for the case $E_j = O\(j^{-1}\)$, it converges slower than the Monte Carlo method. 

\begin{figure}[htbp]
\subfloat[]{\includegraphics[width=0.33\textwidth]{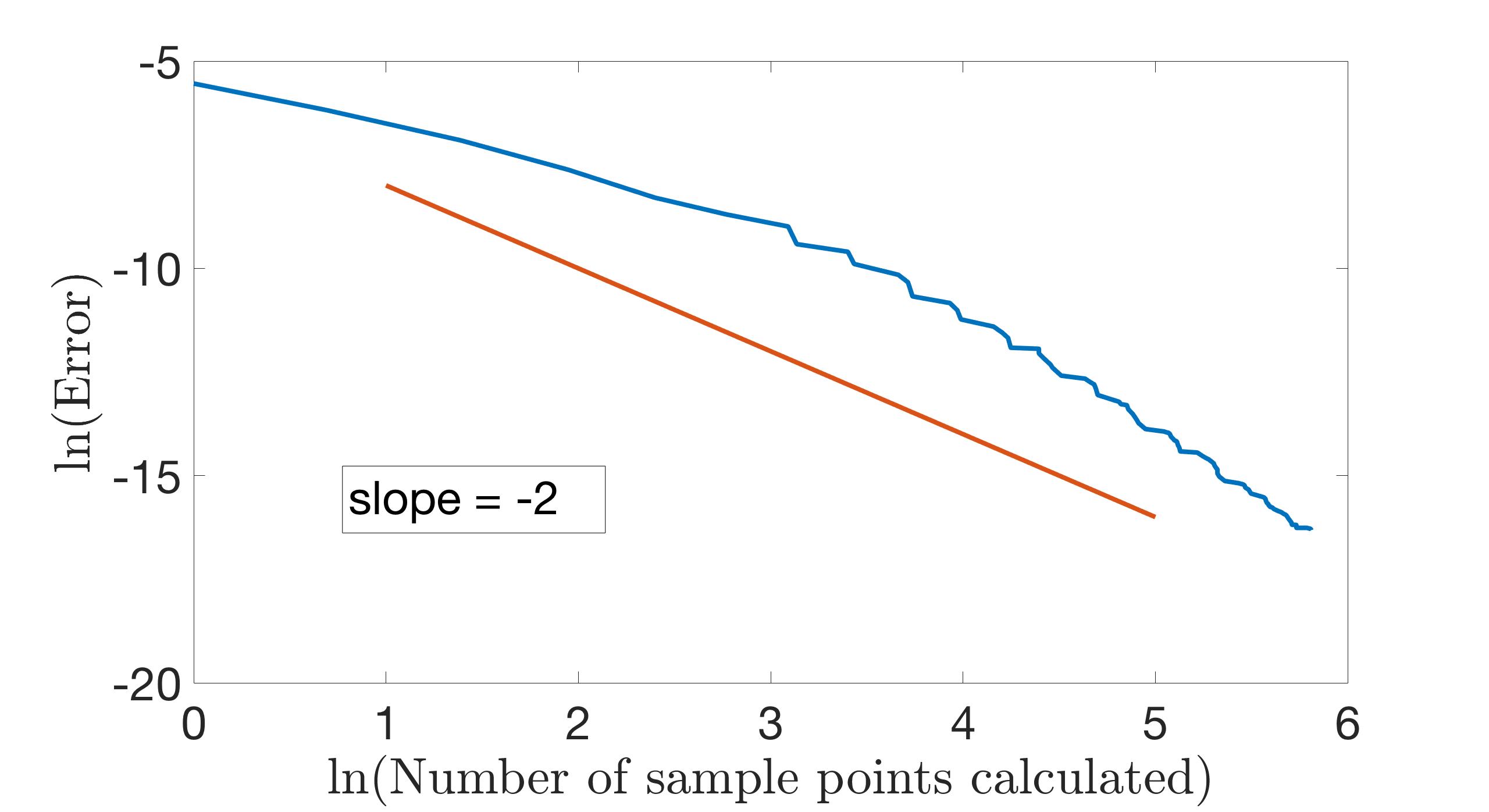}}
\subfloat[]{\includegraphics[width=0.33\textwidth]{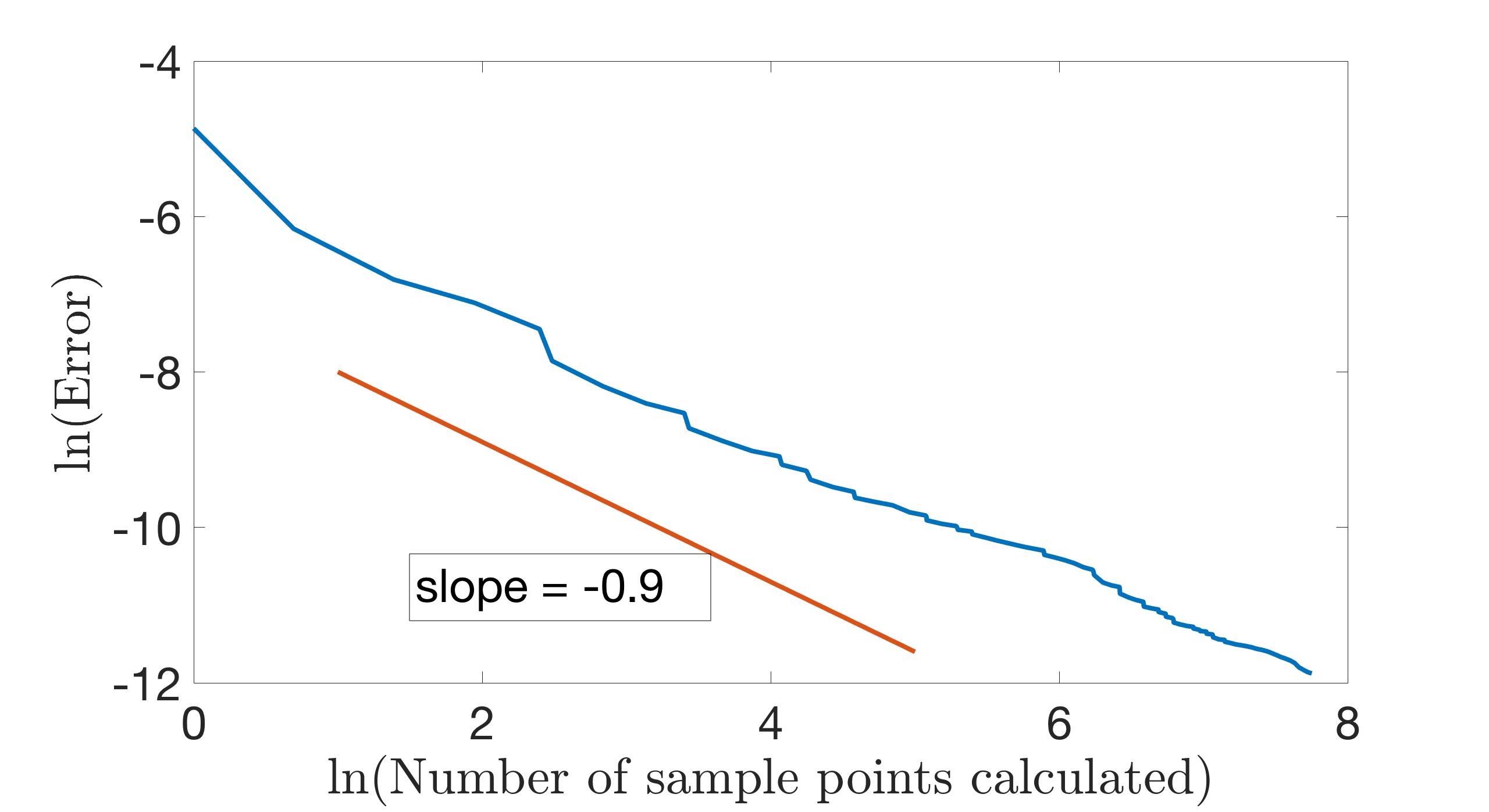}}
\subfloat[]{\includegraphics[width=0.33\textwidth]{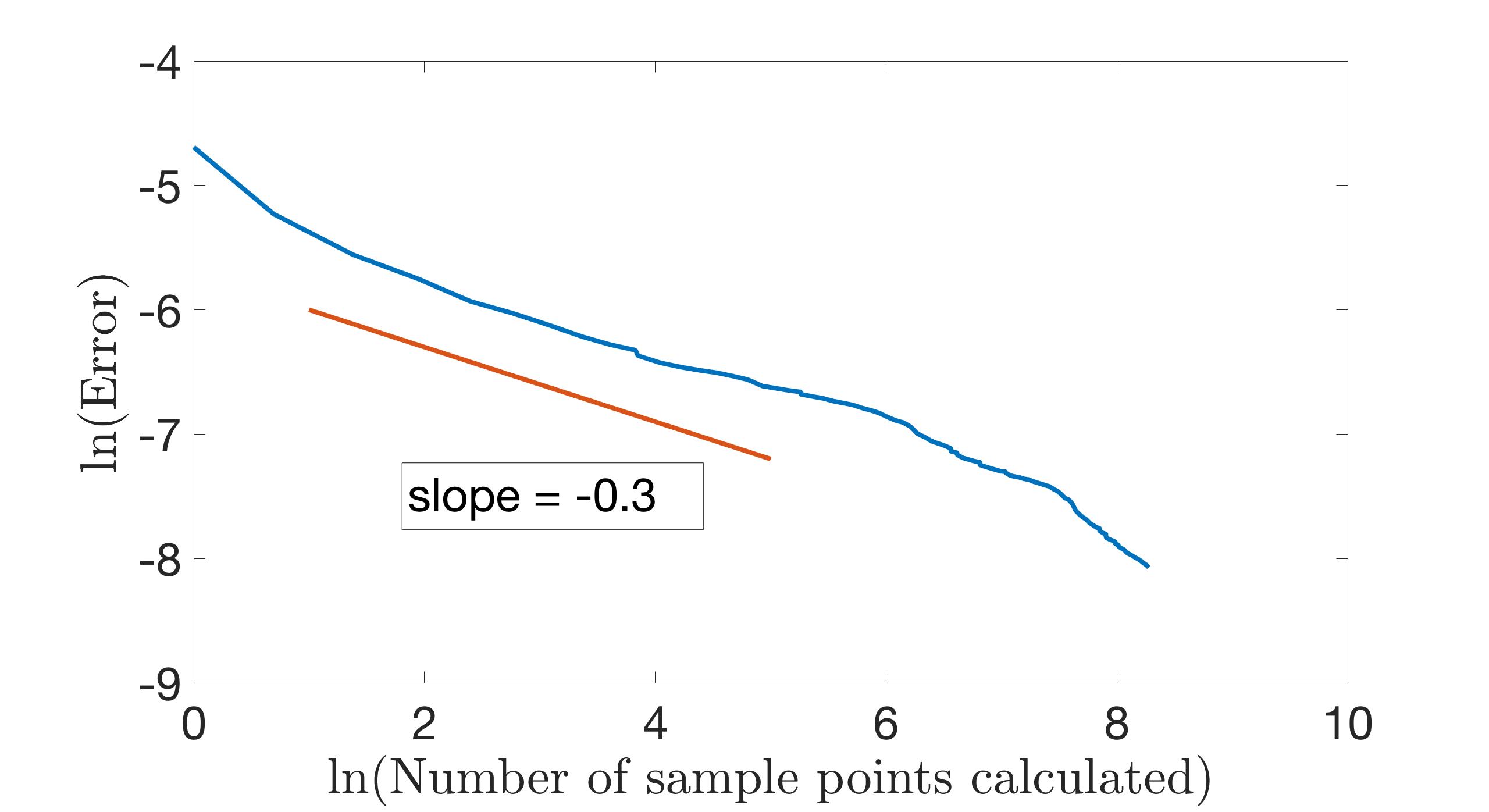}}
\caption{The convergence rate for error defined in (\ref{def of error}) of the approximate solution by the ASPI method  with respect to the number of sample points actually calculated. (a), (b), (c) are the cases of different $E_j$ in (\ref{eq_E}).}
\label{fig_2}
\end{figure}

The right column of Figure \ref{fig_1} shows the projection of the $100$ selected sample points on each dimension. One finds that for all three cases, the number of projection points gets smaller as the dimension gets higher. One also notes that all the dimensions larger than $17$, $66$, $88$ only have one projection point for three cases respectively. This indicates that when $\ll E_j \rl$ decay slower, then more points are projected to higher dimension.

\subsection{Efficiency of the greedy search}
For both methods, one needs to do greedy search for $\nu \in \t{N}(\L_n)$ at the $n$-th step. In Figure \ref{fig_4}, we show that the greedy search is more efficient than just randomly choosing $\nu_{n+1}$ from $\t{N}(\L_n)$. We call the method without greedy search the anisotropic Monte Carlo method. At the $n$-th step, one does the following, 
\begin{algorithm} (Anisotropic MC)
\label{algo: MC}
\begin{itemize}
\item At n-th step, one has $\L_n$, $I_{\L_n}$.
\begin{itemize}
\item [-] Construct $\t{N}(\L_n)$.
\item [-] Uniformly draw $\nu_{n+1}$ from $\t{N}(\L_n)$, compute $f_{\nu_{n+1}} = f(T, \bz_{\nu_{n+1}})$, then construct $\a_{\nu_{n+1}}  = f_{\nu_{n+1}}  - I_{\L_n}(\bz_{\nu_{n+1}})$.
\end{itemize}
\end{itemize}
\end{algorithm}
Since we have shown in Figure \ref{fig_1} that ASPI and the RASPI have almost the same decay rate, so only the decay rate of the RASPI is shown in Figure \ref{fig_4}. Figure \ref{fig_4} shows that for the same number of sample points one calculated, including those in the greedy search, the adaptive greedy search methods (ASPI and RASPI) have faster decay of error compared with the anisotropic Monte Carlo method.  And one notices that as $s$ becomes bigger, that is $\ll \p_j \rl_{W^{1,\infty}_x}$ decays faster,  the benefit one gains from the greedy search becomes more significant. 

\begin{figure}[htbp]
\subfloat[The error for the case of $\displaystyle E(x,\bz) = \frac{\sin(x)}{2} + \sum_{j=1}^{100}\frac{\cos(jx)}{2^j}z_j$]
{\includegraphics[width=0.8\textwidth]{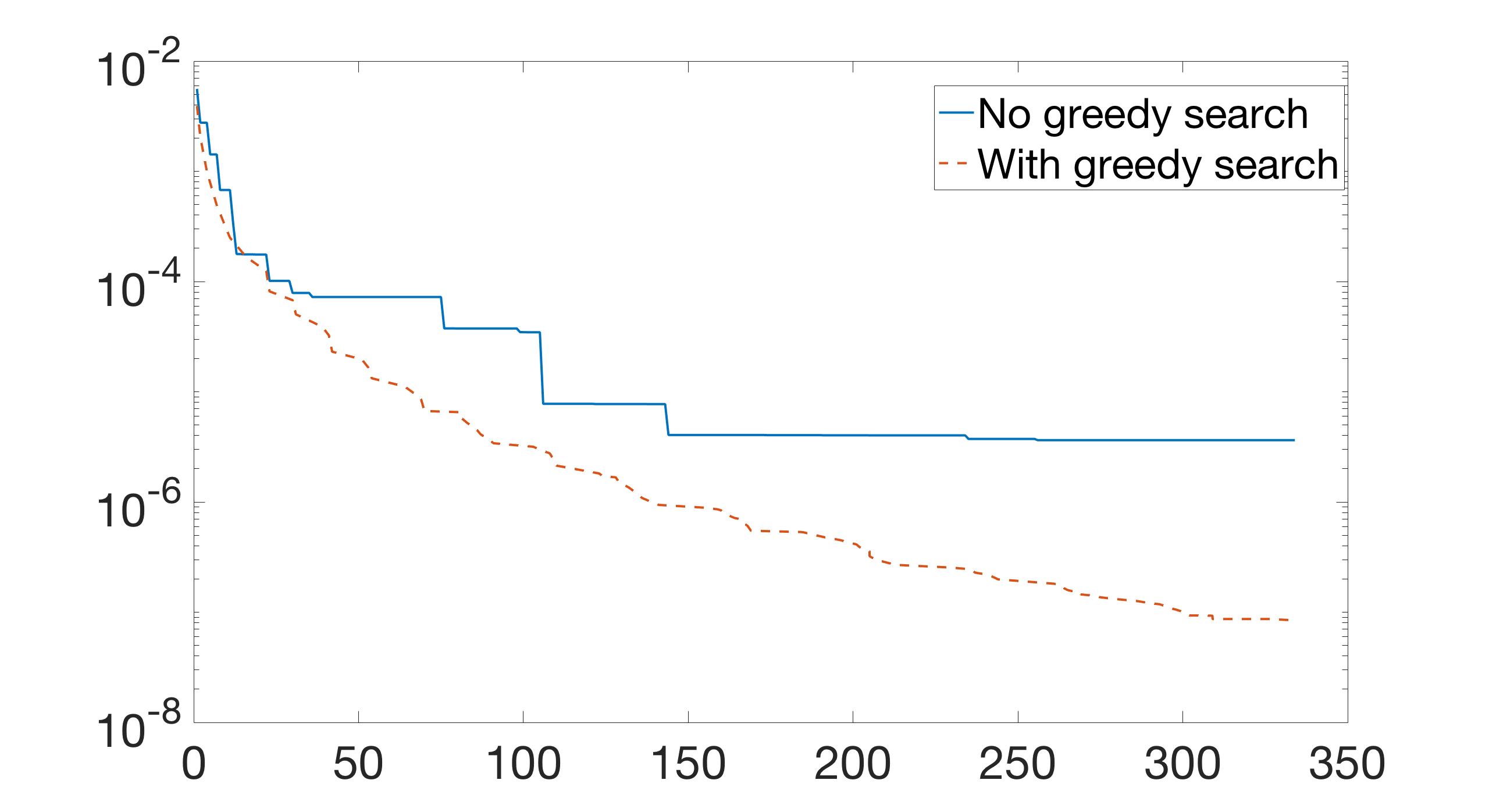}}\\
\subfloat[The error for the case of $\displaystyle E(x,\bz) = \frac{\sin(x)}{2} + \sum_{j=1}^{100}\frac{\cos(jx)}{j^2}z_j$]
{\includegraphics[width=0.8\textwidth]{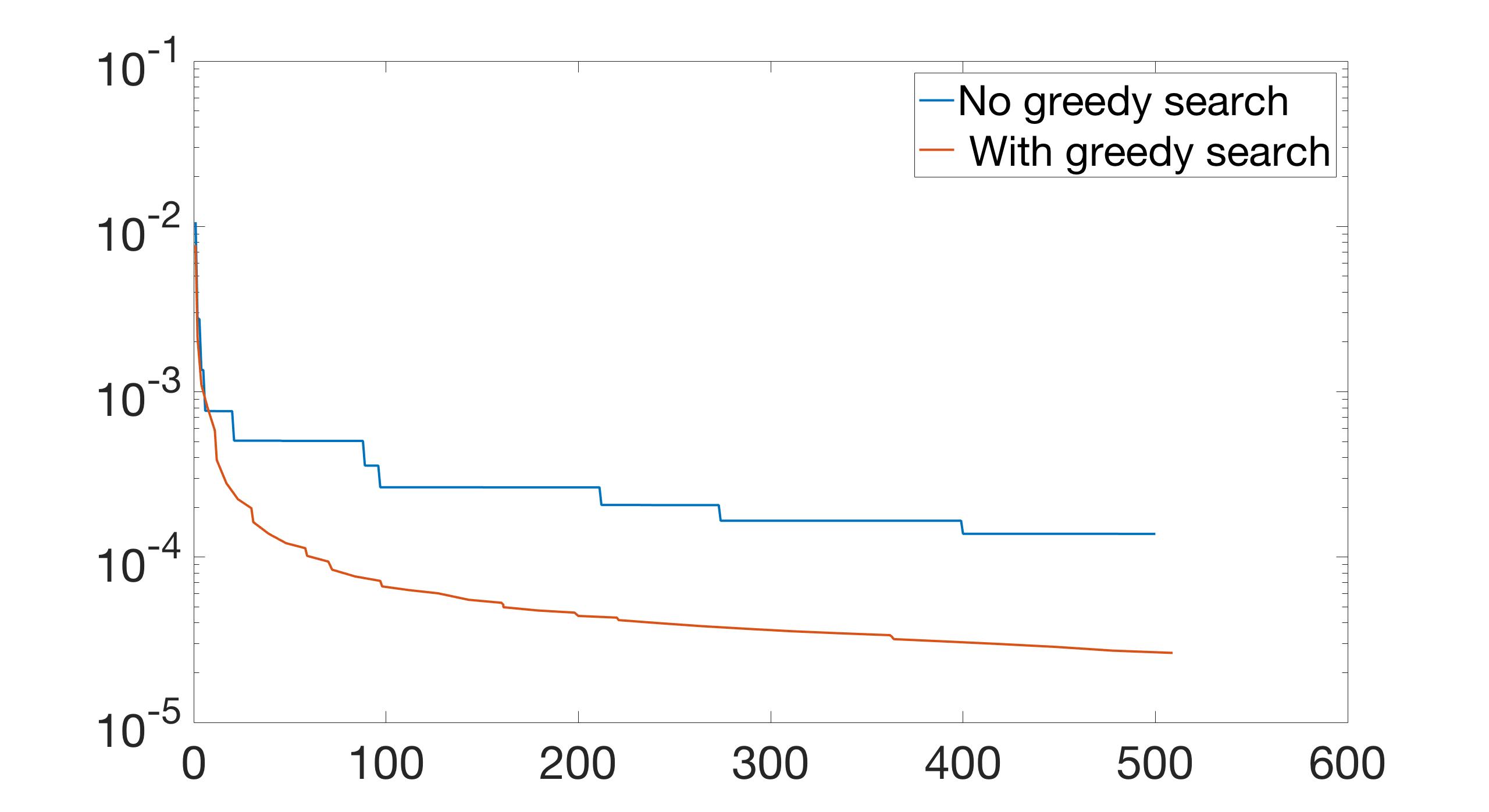}}\\
\subfloat[The error for the case of $\displaystyle E(x,\bz) = \frac{\sin(x)}{2} + \sum_{j=1}^{100}\frac{\cos(jx)}{j}z_j$]
{\includegraphics[width=0.8\textwidth]{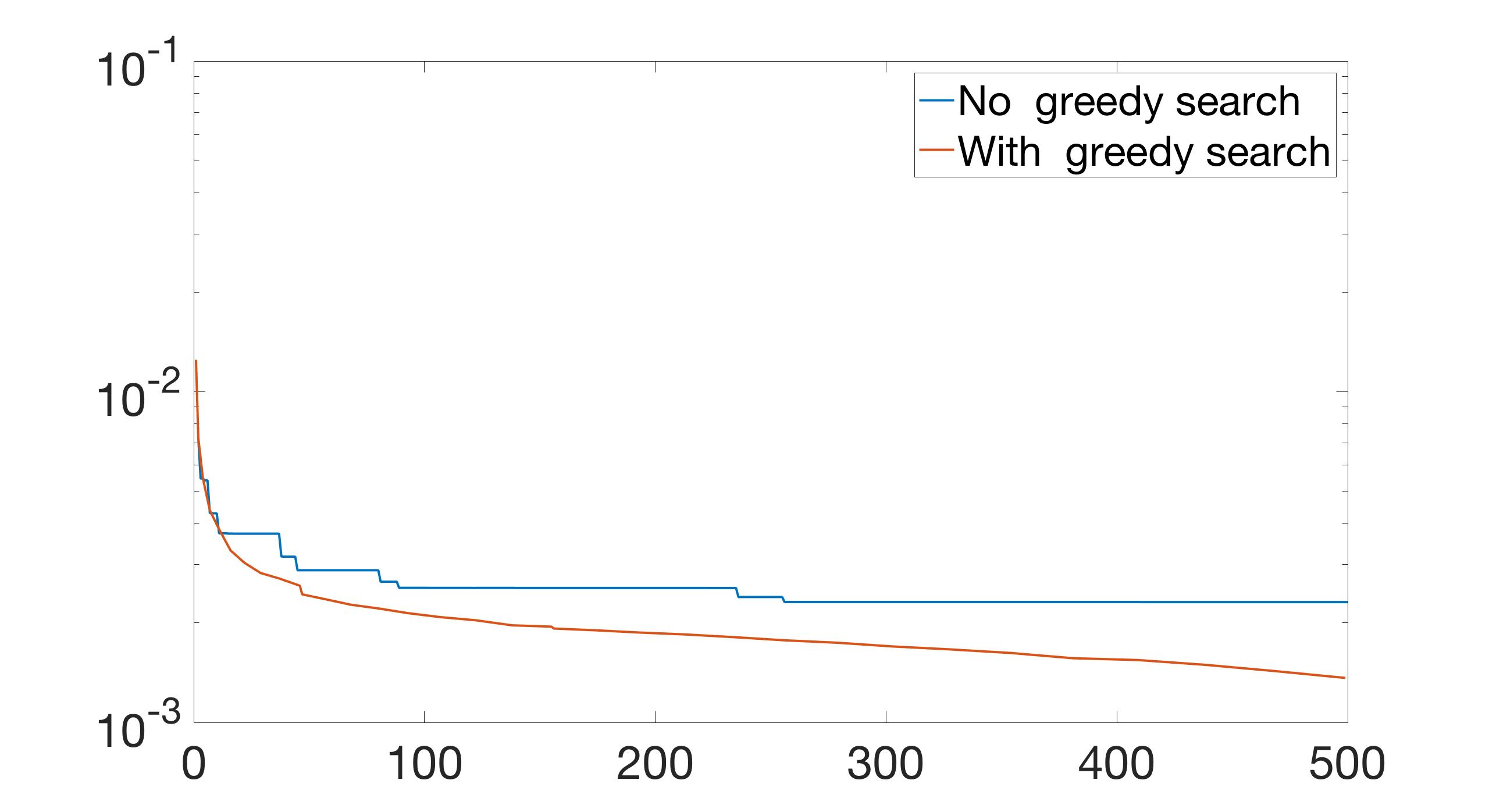}}
\caption{The error of the approximate solution to (\ref{eqn_eq1}) at $t=1$ with different $E_j$ obtained by Algorithm \ref{algo: MC} and the RASPI, where we set $\e = 1$.}
\label{fig_4}
\end{figure}

\subsection{Time dependent electric field}
Now, we will test some examples where the electric field is time dependent, 
\begin{align}
	&a). \ E_j(t,x) =\frac{1}{2^j}\(\cos(jx) + \frac{1}{(1+t)^2}\),\nonumber\\
	&b). \  E_j(x) = \frac{1}{j^2}\(\cos(jx) + \frac{1}{(1+t)^2}\),\nonumber\\
	&c). \ E_j(x) =\frac{1}{j}\(\cos(jx) + \frac{1}{(1+t)^2}\).\label{eq_E of t}
\end{align}
The decay rate of the RASPI is shown in Figure \ref{fig: eq4}. We see the decay rate is slower than the case independent of $t$ as shown in Figure \ref{fig_1}, which is expected from the theoretical result. Since the decay rate in Theorem \ref{conv rate of t} depends on $\l\{\ll \pt^\nu f \rl\r\}_{\nu\in\mF}$, and  $\ll \pt^\nu f \rl= \ll \pt^\nu h \rl$, for $\lv\nu\rv >0$. Notice that the upper bound of $\ll \pt^\nu h \rl$ is
\begin{equation*}
    \ll \pt^\nu h(t) \rl_{L^\infty(U,V)} \leq Q(t)\(|\nu|!\)d^\nu,
\end{equation*}
where 
\begin{equation*}
    Q(t) =  \min\l\{\frac{1}{\e} e^{-\frac{\xi}{\e^2}t}, e^{-\xi t}\r\}2\(\ll h(0) \rl_V+ \sqrt{\frac{\lam \bar{D}}{C_E}}\).
\end{equation*}
For the different cases of $\E_j$ in (\ref{eq_E}) and (\ref{eq_E of t}), since $\lim_{t\to \infty} E_j(t,x)$ in (\ref{eq_E of t}) is equal to $\E_j$ in (\ref{eq_E}), so the only difference in the upper bound of $\ll \pt^\nu h \rl$ is on $\bar{D}$ appearing in $Q(t)$ and defined in (\ref{def of barD}). Specifically, $\bar{D} = 0$ for the time independent case and $\bar{D} > 0 $ for the time dependent case. So the time dependent case should have slower decay rate compared to the time independent case. 

\begin{figure}[htbp]
\subfloat[]{\includegraphics[width=0.33\textwidth]{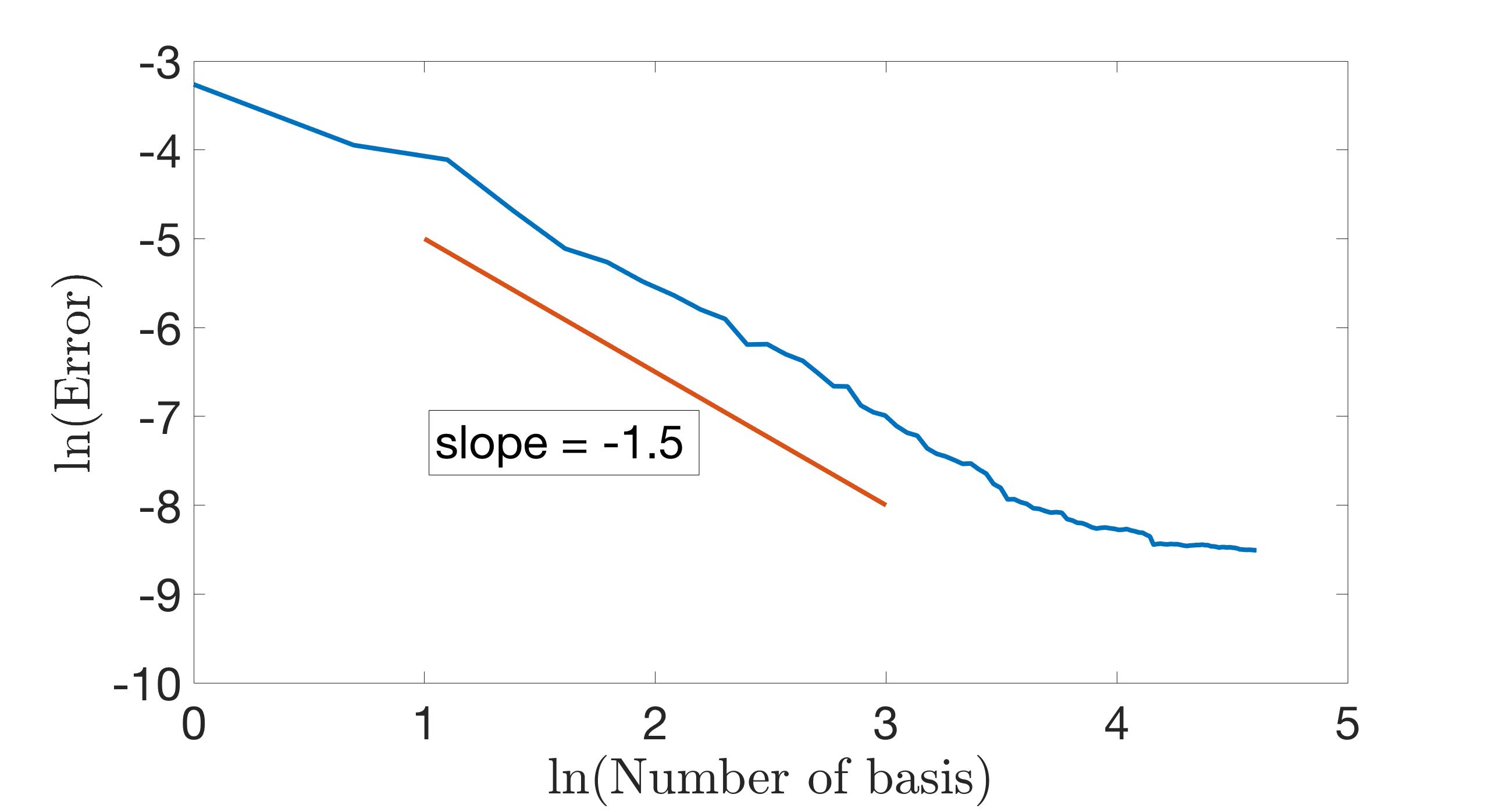}}
\subfloat[]{\includegraphics[width=0.33\textwidth]{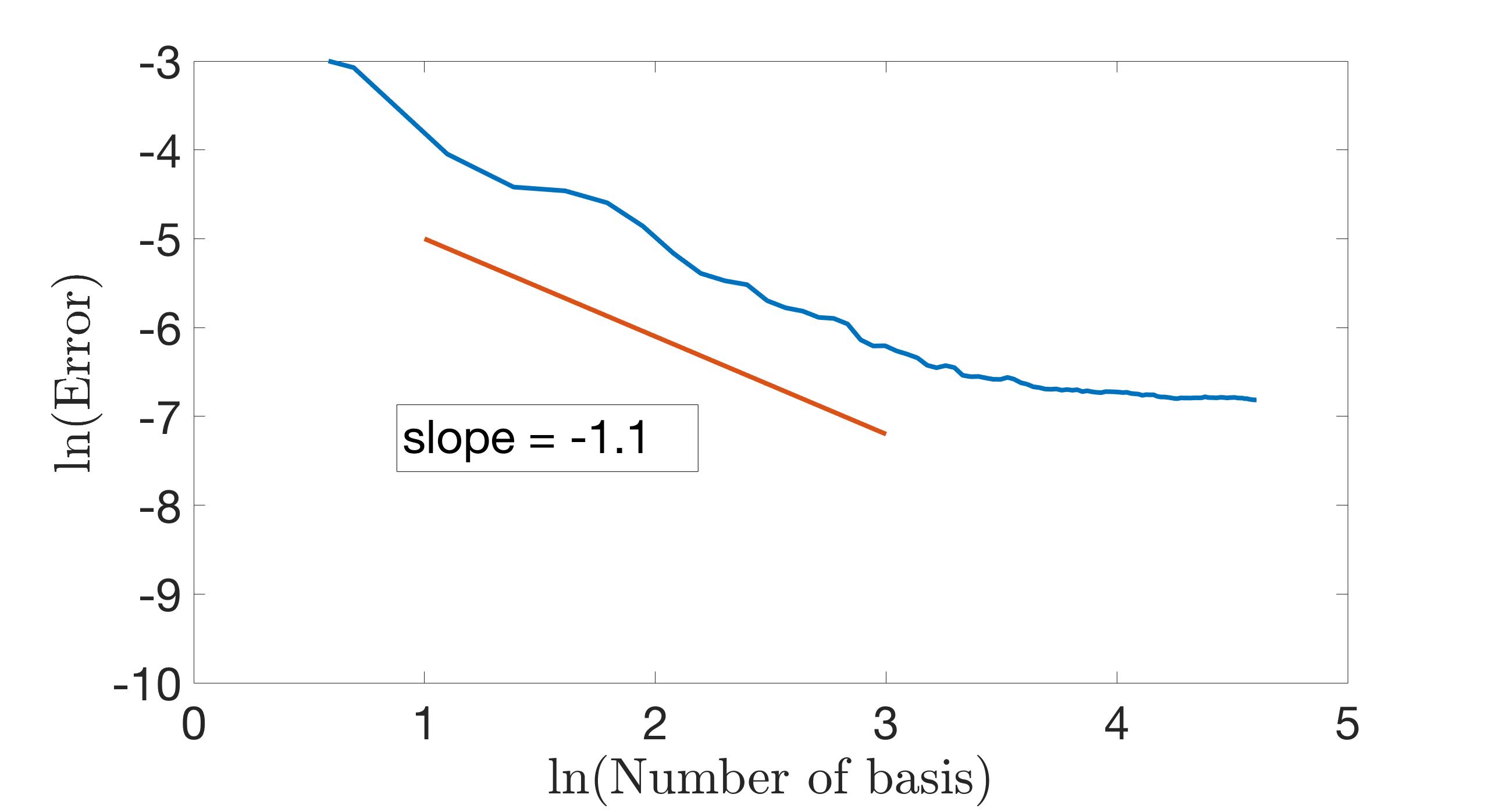}}
\subfloat[]{\includegraphics[width=0.33\textwidth]{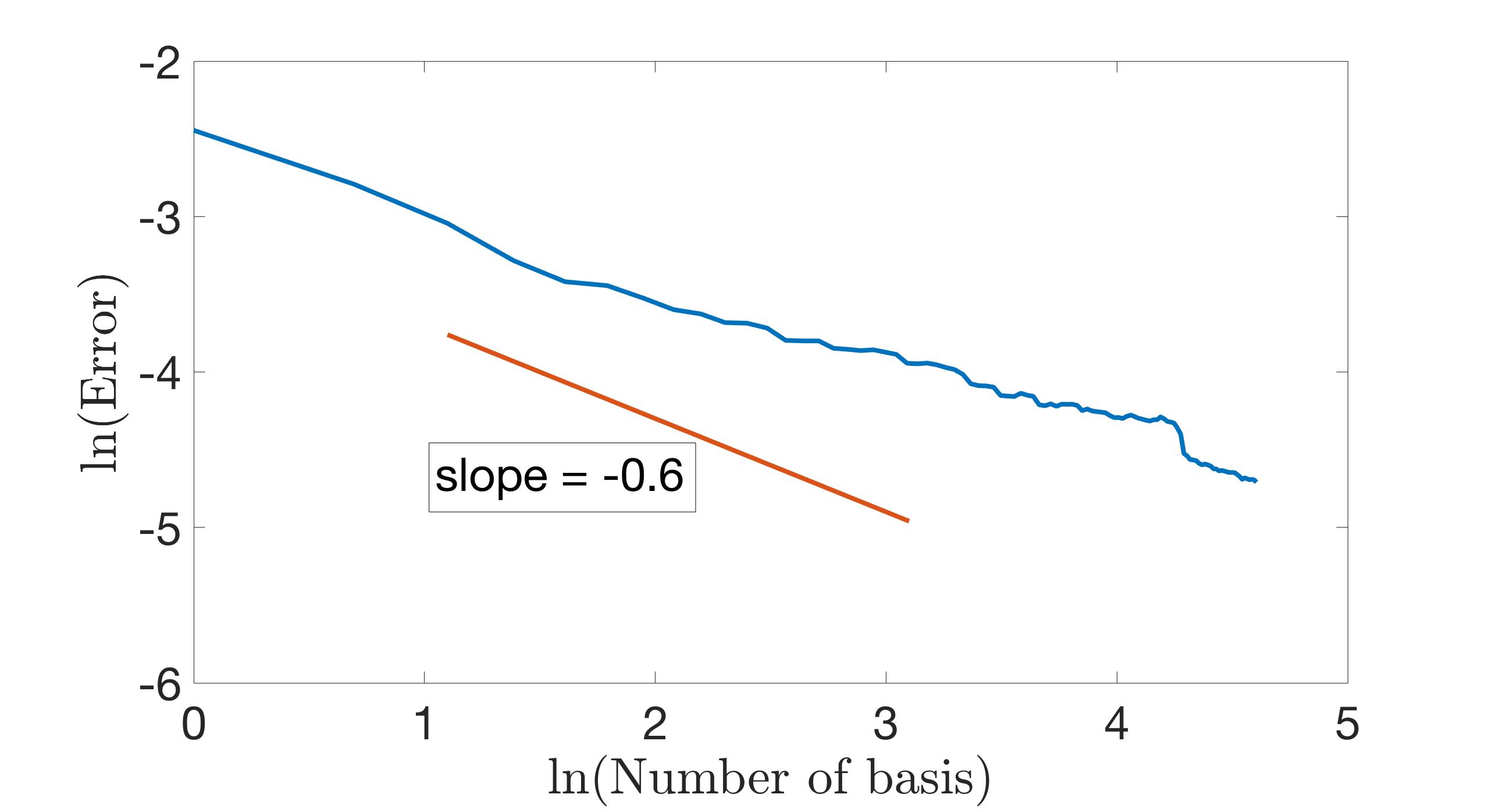}}
\caption{The convergence rate for error of the approximate solution by the RASPI with respect to the number of sample points. (a), (b), (c) are the cases of different $E_j$ in (\ref{eq_E of t}).}
\label{fig: eq4}
\end{figure}

\subsection{$\e$ dependency}
Finally, we will check the error dependence on different $\e$. The error of the RASPI method  is shown in Figure \ref{fig: eq5}. The decay rate for different $\e$ is similar, but smaller $\e$ yields smaller error. Our explanation on this is that for smaller $\e$,  the solution is closer to the global Maxwellian (\ref{def of M}), which is deterministic thus the sampling error is less relevant. 
\begin{figure}[htbp]
{\includegraphics[width=\textwidth]{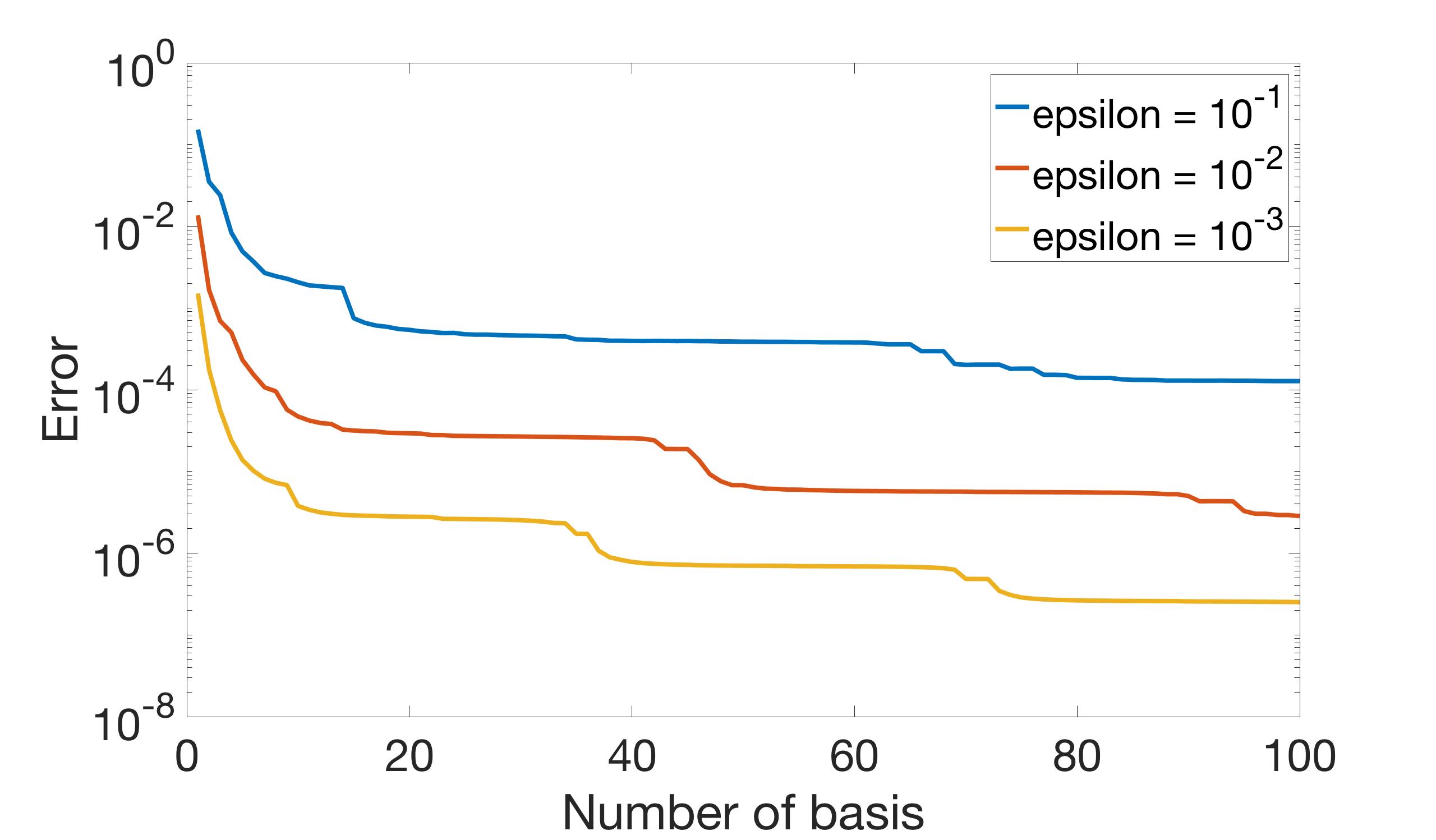}}
\caption{The error of the RASPI with respect to the number of sample points, with $E_j$ defined in (\ref{eq_E}) b).}
\label{fig: eq5}
\end{figure}

\section{Conclusion}
\label{Concl}
In this paper, we first showed theoretically that if the forcing term $E(t,x,z) = \bar{E}(t,x)+\sum_{j\geq}E_j(t,x)z_j$ has anisotropic property in random space,  converges to the steady state  $E^\infty(x,z)$ fast enough and is
bounded above,  then the best N approximation based on the Legendre basis converges to the solution of the Vlasov-Fokker-Planck equation in random space
with an error of $O\(N^{-\(\frac{1}{p} - \frac{1}{2}\)}\)$, for $\(\ll E_j\rl_{W^{1,\infty}}\)_{j\geq1} \in \ell^p$. 

Numerically, we develop the residual based adaptive sparse polynomial
interpolation (RASPI) method based on the adaptive sparse polynomial
interpolation (ASPI). We show through numerical examples that RASPI
converges to the solution independent of dimension of the random variables. We also show that for  general linear kinetic equation, or equivalently, for general time dependent and implicit scheme,  the ratio of the computation cost of the ASPI to the RASPI is $O(\d^{-5+2/s})$ for $\frac{1}{2}\leq s\leq 1$ and $O(\d^{-3})$ for $s\geq1$,  which means that the faster $\ll E_j\rl_{W^{1,\infty}}$ decays, the more the RASPI saves. 

There are still several open questions worthy of study in the future.
For example, the rigorous convergence rate of the RASPI and the ASPI method
remain to be established. Another important problem is whether nonlinear kinetic equations, such as the Boltzmann equation and Vlasov-Poisson-Fokker-Planck equations  with high dimensional uncertain parameters, can be solved by
these methods.

\section*{Acknowledgement}
The first author was partially supported by NSF grants DMS-1522184, DMS-1819012, DMS-1107291: RNMS KI- Net, and NSFC grant No. 31571071. The third author has received funding of the Alexander von Humboldt-Professorship
program,   the European Research Council (ERC) under the
European Union's Horizon 2020 research and innovation program (grant agreement NO. 694126-
DyCon9,  the Air Force Office of Scientific
Research under Award NO: FA9550-18-1-0242, the Grant MTM2017-92996-C2-1-R COSNET of MINECO (Spain),  the ELKARTEK project KK-2018/00083 ROAD2DC of the Basque Government, and  the European Unions Horizon 2020 research and innovation program under the Marie
Sklodowska-Curie grant agreement NO. 765579-ConFlex.

\appendix
\section*{Appendices}
\addcontentsline{toc}{section}{Appendices}
\renewcommand{\thesubsection}{\Alph{subsection}}

\subsection{The proof of Lemma \ref{energy est of t}}
\label{proof of energy est of t}
\subsubsection{The proof of the time independent random electric field}
\setcounter{equation}{0}
\setcounter{theorem}{0}
\renewcommand\theequation{A.\arabic{equation}}
\renewcommand\thetheorem{A.\arabic{theorem}}
We will first prove the case when the random electric field is time independent, that is, $E(t,x,z)\equiv E(x,z)$.
\begin{lemma}
\label{energy est}
Under Condition \ref{cond of E of t}, for  $\forall \bz \in U$,  the following estimates hold, 
\begin{align}
	&|\nu| = 0:  \pt_t G^0_i  + \frac{\eta}{\e}\ll h\rl_V^2  \leq  0;\label{G}\\
	&|\nu| > 1: \pt_tG_i^\nu +  \frac{\eta}{\e}\ll \pt^{\nu}h\rl_V^2\leq \frac{2}{\lam\e}\sum_{\nu_j\neq0} \nu_j^2C_j  \ll \pt^{\nu-e_j}h \rl_V^2,\label{G_nu}
\end{align}
where $i = 1,2$, $\eta = \frac{C_s}{10}$, and 
\begin{align}
	&\frac{\e}{2\lam} \ll \pt^{\nu}h \rl_V^2  \leq G_1^\nu \leq \frac{3}{2\lam\e} \ll \pt^{\nu}h \rl_V^2,\qd \frac{1}{2\lam\e} \ll \pt^{\nu}h \rl_V^2  \leq G_2^\nu \leq \frac{3}{2\lam\e} \ll \pt^{\nu}h \rl_V^2 .\label{G_2}
\end{align}
\end{lemma}

\begin{proof}
For  $E(t,x,\bz) \equiv E(x, \bz)$, the microscopic equation (\ref{micro eq with t}) for the perturbative solution $h(t,x,v, \bz)$ is simplified to,
\begin{align}
	\e\pt_th + v\pt_xh - \frac{1}{\e} \mL h = E\(\pt_v - \frac{v}{2}\)h,
	\label{micro eq}
\end{align}
where $\mL$ is the linearized Fokker Planck operator defined in (\ref{linearized FP}) that satisfies the local coercivity property as in (\ref{coercivity}). 
If one multiplies $\sM$ and $v\sM$ to (\ref{micro eq}) respectively, and integrates them  over $v$, then one gets the {\it{macroscopic equations}},
\begin{empheq}[left=\empheqlbrace]{align}
&\e\pt_t\s + \pt_xu = 0\label{continuity}\\
&\e\pt_tu + \pt_x\s + \int v^2\sM (1-\Pi)\pt_xh dv + \frac{1}{\e}u = -E\s.
	\label{macro eq}
\end{empheq}
The first equation is the perturbative continuity equation, while the second one is the perturbative momentum equation. Notice the operators $\Pi$ and $1-\Pi$ are perpendicular to each other in $L^2_{x,v}$, that is,
\begin{align}
	\ll h\rl^2 = \ll \Pi h \rl^2 + \ll (1-\Pi) h\rl^2 = \ll \s \rl^2 + \ll (1-\Pi) h\rl^2.
\end{align}

If one takes $\pt^\nu$ and $\pt^\nu\pt_x$  to (\ref{micro eq}), and multiplies $\pt^\nu_\bz h$ and $\pt^\nu\pt_xh$ respectively, then integrates them over $x,v$ and adds the two equations together, one has
\begin{align}
	&\frac{\e}{2}\pt_t \ll \pt^{\nu}h \rl_V^2 +\frac{\lam}{\e} \ll (1-\Pi)\pt^{\nu}h\rl_{V,\o}^2
	\nonumber\\
	\leq& \underbrace{\la \pt^\nu\(E\(\pt_v-\frac{v}{2}\)h \), \pt^\nu h\ra}_{I} 
	+\underbrace{\la \pt^\nu\pt_x\(E\(\pt_v-\frac{v}{2}\)h \), \pt^\nu \pt_xh\ra}_{II}, \label{eq: proc_1}
\end{align}
where the second term $II$ comes from the hypocoercivity of $\mL$ in (\ref{coercivity}). If one takes $\pt^\nu$ to (\ref{macro eq}) , and multiplies $\pt^\nu \pt_x\s$, then integrates it over $x,v$, one has
\begin{align}
	&\pt_t \(\e\la \pt^{\nu}u, \pt^{\nu}\pt_x\s \ra + \frac{1}{2}\ll \pt^{\nu}\s\rl^2\)  + \frac{1}{2}\ll \pt^{\nu}\pt_x\s\rl^2\nonumber\\
	 \leq& \ll  \pt^{\nu}\pt_xu\rl^2 + \frac{1}{2}\ll (1-\Pi)\pt^{\nu}h \rl_{V,\o}^2 - \underbrace{ \la \pt^{\nu}(E\s), \pt^{\nu}\pt_x\s\ra}_{III} .
	 \label{eq: proc_2}
\end{align}
Here the first term on the LHS and the first term on the RHS come from
\begin{align}
	&\la \pt_t \pt^{\nu}u,  \pt^{\nu}\pt_x\s \ra = \pt_t\la  \pt^{\nu}u,  \pt^{\nu}\pt_x\s\ra - \la  \pt^{\nu}u, \pt_t \pt^{\nu}\pt_x\s \ra \nonumber\\
	=&  \pt_t\la  \pt^{\nu}u,  \pt^{\nu}\pt_x\s\ra + \frac{1}{\e}\la  \pt^{\nu}u, \pt^{\nu}\pt_x^2u \ra =  \pt_t\la  \pt^{\nu}u,  \pt^{\nu}\pt_x\s\ra - \frac{1}{\e}\ll  \pt^{\nu}\pt_xu\rl^2,
\end{align}
where the second equality is because of the continuity equation (\ref{continuity}). The second term on the LHS of (\ref{eq: proc_2}) is because of 
\begin{align}
	&\frac{1}{\e}\la  \pt^{\nu}u,  \pt^{\nu}\pt_x\s \ra = \la -\frac{1}{\e} \pt^{\nu}\pt_xu,  \pt^{\nu}\s \ra = \la \pt_t \pt^{\nu}\s,  \pt^{\nu}\s \ra = \frac{1}{2}\pt_t\ll \pt^{\nu}\s \rl^2. 
\end{align}
Furthermore since,
\begin{align}
	&\la  \pt^{\nu}\pt_x\s, \pt^{\nu}\pt_x\s\ra + \la \int v^2(1-\Pi)\pt_x \pt^{\nu}h\sM\,dv ,  \pt^{\nu}\pt_x\s \ra\geq \frac{1}{2}\ll \pt^{\nu}\pt_x\s \rl^2-\frac{1}{2}\ll (1-\Pi)\pt^{\nu}\pt_xh \rl^2_\o,
\end{align}
this gives the third term on the LHS and the second term on the RHS of (\ref{eq: proc_2}). 

Furthermore, integrate (\ref{continuity}) over $x$, by the periodic boundary condition, one has, 
\begin{equation*}
	\pt_t \int \s dx = 0.
\end{equation*}
Since from (\ref{initial cond}), we know that
\begin{equation*}
	\int \s(0,x,z) dx = \int h\sM dxdv = 0,
\end{equation*}
so
\begin{equation*}
	\int \s(t,x,z) dx = \int \s(0,x,z) dx  = 0.
\end{equation*}
Similar equality can be obtained for $\pt^\nu\s$. Therefore, one can apply the Poincare inequality (\ref{sobolev const}) to $\ll \pt_x\s \rl^2$ to get, 
\begin{equation*}
	\ll \pt^\nu\pt_x\s \rl^2 \geq C_s\ll \pt^\nu\s \rl_V^2.
\end{equation*}
By adding $\theta_i(\ref{eq: proc_1}) + \frac{1}{2\e}(\ref{eq: proc_2})$, and the fact that 
\begin{equation*}
    \ll  \pt^{\nu}\pt_xu\rl^2 \leq \ll (1-\Pi)\pt^{\nu}\pt_xh \rl^2\leq \ll (1-\Pi)\pt^{\nu}h \rl_{V,\o}^2,
\end{equation*}
one has, 
\begin{align}
	&\pt_t G_i^{\nu} +\frac{\lam\theta_i}{\e} \ll (1-\Pi)\pt^{\nu}h\rl_{V,\o}^2 + \frac{ C_s}{4\e}\ll \pt^{\nu}\s\rl_V^2
	\leq  \frac{3}{4\e}\ll (1-\Pi)\pt^{\nu}h \rl_{V,\o}^2 + \theta_i\(I+II\)-\frac{1}{2\e}III, \label{est}
\end{align}
where $C_s$ comes from the Poincare inequality (\ref{sobolev const}). In order to find an estimate for $G^\nu_i$, one needs to estimate terms $\theta_i\(I+II\)-\frac{1}{2\e}III$.  First notice that 
\begin{align}
	&\(\pt_v + \frac{v}{2}\)\pt^\nu \pt_x^i (\Pi h) = -\frac{v}{2}\(\pt^\nu \pt_x^i \s\)\sM + \frac{v}{2}\(\pt^\nu \pt_x^i \s\)\sM = 0, \qd \text{for }i = 0, 1, 
\end{align}
which implies that terms $I$ and $II$ can be simplified to, for $i = 0,1$,
\begin{align}
	I, II =& -\la \pt^\nu\pt_x^i(Eh), \(\pt_v + \frac{v}{2}\)\pt^\nu \pt_x^ih\ra =  -\la \pt^\nu\pt_x^i(Eh), \(\pt_v + \frac{v}{2}\)\(\pt^\nu\pt_x^i( \Pi h) + (1-\Pi)\pt^\nu \pt_x^ih\)\ra  \nonumber\\
	=&  -\la \pt^\nu\pt_x^i(Eh), \(\pt_v + \frac{v}{2}\) (1-\Pi)\pt^\nu\pt_x^i h\ra. \label{eq: proc_11}
\end{align}
Another inequality that will be used frequently later is that  for $i = 0,1$,
\begin{align}
	&\ll \(\pt_v + \frac{v}{2}\)\(1-\Pi\) \pt^\nu  \pt_x^i h\rl^2 \nonumber\\
	=& \ll \pt_v(1-\Pi)\pt^\nu  \pt_x^i h\rl^2 + \frac{1}{4}\ll v (1-\Pi)\pt^\nu  \pt_x^i h \rl^2 + \int_\O \frac{v}{2}\pt_v\((1-\Pi)\pt^\nu  \pt_x^i h\)^2dxdv  \nonumber\\
	=& \ll \pt_v(1-\Pi)\pt^\nu  \pt_x^i h\rl^2 + \frac{1}{4}\ll v (1-\Pi)\pt^\nu  \pt_x^i h \rl^2 - \frac{1}{2} \ll (1-\Pi)\pt^\nu  \pt_x^i h \rl^2
	\leq \ll (1-\Pi)\pt^\nu  \pt_x^i h \rl^2_\o. \label{proc_12}
\end{align}

Based on (\ref{eq: proc_11}) and  (\ref{proc_12}), we will bound the term $\theta_i\(I+II\)-\frac{1}{2\e}III$ for the cases $|\nu| = 0, |\nu|>1$ respectively.
Firstly, for the case $|\nu| = 0$, 
\begin{align}
	&\theta_i\(I+II\)-\frac{1}{2\e}III \nonumber\\
	=& \thi \la Eh, \(\pt_v + \frac{v}{2}\)(1-\Pi) h\ra + \thi\la \pt_x\(Eh\), \(\pt_v + \frac{v}{2}\)(1-\Pi) h\ra - \frac{1}{2\e}\la E\s, \pt_x\s\ra \nonumber\\
	\leq& \frac{\thi}{2}\ll E \rl_{L^\infty_x}\(\e\ll h \rl_V^2 + \frac{1}{\e}\ll (1-\Pi) h\rl_{V,\o}^2\) 
	+ \frac{\thi}{2}\ll \pt_xE \rl_{L^\infty_x}\(\e\ll h \rl^2 + \frac{1}{\e}\ll (1-\Pi) \pt_xh\rl_\o^2\) \nonumber\\
	&+ \frac{1}{4\e} \ll E \rl_{L^\infty_x}\(\ll\s \rl^2 + \ll \pt_x\s \rl^2\). \label{proc_13}
\end{align}
Since $\ll E \rl_{L^\infty(U,W^\infty_x)} \leq C_E$, and $\ll h \rl^2_V = \ll \s \rl_V^2 + \ll (1-\Pi)h \rl_V^2$, one can further simplify (\ref{proc_13}) to 
\begin{align}
	&(\ref{proc_13})
	\leq C_E\(\(\frac{\e\thi}{2} + \frac{1}{4\e}\)\ll \s \rl_V^2 + \frac{\thi}{\e}\ll (1-\Pi) h\rl_{V,\o}^2\). \label{proc_14}
\end{align}
Plug the above estimate into (\ref{est}), one has, 
\begin{align}
\label{appendix_B_1}
	&\pt_t G_i^0 +\frac{\lam\theta_i}{\e} \ll (1-\Pi)h\rl_{V,\o}^2 + \frac{ C_s}{4\e}\ll \s\rl_V^2
	\nonumber\\
	\leq&  \frac{3}{4\e}\ll (1-\Pi)h \rl_{V,\o}^2 + C_E\(\(\frac{\e\thi}{2} + \frac{1}{4\e}\)\ll \s \rl_V^2 + \frac{\thi}{\e}\ll (1-\Pi) h\rl_{V,\o}^2\),
\end{align}
which implies,
\begin{align}
	&\pt_t G_i^{0} + C^0_h \ll (1-\Pi)h\rl_{V,\o}^2 + C^0_\s\ll \s\rl_V^2
	\leq 0,\label{final_1}
\end{align}
where 
\begin{align}
	&C^0_h = \frac{\lam\theta_i}{\e} -\frac{3}{4\e} - \frac{C_E\thi}{\e},\qd C^0_\s =  \frac{ C_s}{4\e} -  C_E\(\frac{\e\thi}{2} + \frac{1}{4\e}\).
\end{align}
Since $C_E\leq \frac{\lam C_s}{8} \leq \frac{\lam}{16}$, so $\lam - C_E  \geq \frac{15}{16}\lam$, which implies $\(\lam - C_E\)\frac{\thi}{\e} \geq \frac{4}{\lam}$ for both $i = 1,2$. Therefore $C_h^0 \geq \frac{3}{\e}$. Since $C_s - C_E \geq \frac{7}{8}C_s$, and $\e C_E\thi \leq \frac{\e C_s}{7}$ for both $i = 1,2$, therefore $C_\s^0 \geq \frac{7C_s}{32\e} - \frac{\e C_s}{14} \geq \frac{C_s}{10\e}$. Hence plug back $C^0_h, C^0_\s \geq \frac{C_s}{10\e}$ to (\ref{final_1}), one has,
\begin{align}
	&\pt_t G_i^{0} + \frac{\eta}{\e}\ll h\rl_{V,\o}^2 \leq 0,\qd \text{for } \eta = \frac{C_s}{10},
\end{align}
which is exactly (\ref{G}) in Lemma \ref{energy est}. 

For the case where $|\nu| > 0 $, based on (\ref{eq: proc_11}),  (\ref{proc_12}) and the definition of $E = \bar{E}(x) + \sum_{j\geq1}E_j(x)z_j$, one can bound $\theta_i\(I+II\)-\frac{1}{2\e}III$ by
\begin{align}
	&\theta_i\(I+II\)-\frac{1}{2\e}III\nonumber\\
	\leq &\thi\la E\pt^{\nu}h, \(\pt_v+\frac{v}{2})(1-\Pi\)\pt^{\nu}h\ra
	+  \thi\la \pt_x\(E\pt^{\nu}h\), \(\pt_v+\frac{v}{2}\)(1-\Pi)\pt^{\nu}\pt_xh\ra 
	- \frac{1}{2\e} \la E\pt^{\nu}\s, \pt^{\nu}\pt_x\s\ra\nonumber\\
	&+\sum_{\nu_j\neq0}\(\thi\la \nu_jE_j\pt^{\nu-e_j}h, (\pt_v+\frac{v}{2})(1-\Pi)\pt^{\nu}h\ra
	+ \thi \la \pt_x\( \nu_jE_j\pt^{\nu-e_j}h\), (\pt_v+\frac{v}{2})(1-\Pi)\pt^{\nu}\pt_xh\ra \r.\nonumber\\
	&\l.-\frac{1}{2\e} \la  \nu_jE_j\pt^{\nu-e_j}\s , \pt^{\nu}\pt_x\s\ra\)\label{proc_15}.
\end{align} 
Since the three terms in the second line of the above equation is similar to the case where $|\nu| = 0$, hence one can get similar estimates as in (\ref{proc_14}). In addition,  by assumption $\ll E_j \rl_{L^\infty(U,W^{1,\infty}_x)} \leq C_j$,  one has
\begin{align}
	&(\ref{proc_15})\nonumber\\
	\leq & C_E\( \(\frac{\e\thi}{2}+\frac{1}{4\e}\)\ll \pt^{\nu}\s \rl_V^2+ \frac{\thi}{\e}\ll (1-\Pi)\pt^{\nu}h \rl_{V,\o} ^2\)\nonumber\\
	&+\frac{\thi}{2}\sum_{\nu_j\neq0} C_j\(\e\nu_j^2\ll \pt^{\nu-e_j}h \rl_V^2  + \frac{1}{\e}\ll (1-\Pi)\pt^{\nu}h \rl_{V,\o}^2\) +\frac{1}{4\e}\sum_{\nu_j\neq0}C_j\(\nu_j^2\ll\pt^{\nu-e_j}\s\rl^2 + \ll \pt^{\nu}\pt_x\s \rl^2 \)\nonumber\\
	\leq & C_E\( \frac{1}{2}\(\e\thi+\frac{1}{\e}\)\ll \pt^{\nu}\s \rl_V^2+ \frac{3\thi}{2\e}\ll (1-\Pi)\pt^{\nu}h \rl_{V,\o} ^2\)
	+\sum_{\nu_j\neq0} \nu_j^2C_j\(\frac{\e\thi}{2} + \frac{1}{4\e}\)\ll \pt^{\nu-e_j}h \rl_V^2  ,\label{proc_24}
\end{align}
where the second inequality is because that $\sum_{j\geq 1}C_j \leq C_E$. Hence plug (\ref{proc_24}) into (\ref{est}), one has, 
\begin{align}
	&\pt_t G^\nu_i + C^\nu_h \ll (1-\Pi)\pt^\nu h\rl_{V,\o}^2 + C^\nu_\s\ll \pt^\nu\s\rl_V^2
	\leq \sum_{\nu_j\neq0} \nu_j^2C_j\(\frac{\e\thi}{2} + \frac{1}{4\e}\)\ll \pt^{\nu-e_j}h \rl_V^2,
\end{align}
where 
\begin{align}
	&C^\nu_h = \frac{\lam\theta_i}{\e} -\frac{3}{4\e} - \frac{3C_E\thi}{2\e} \geq \frac{\eta}{\e},\qd C^\nu_\s =  \frac{ C_s}{4\e} -  \frac{C_E}{2}\(\e\thi + \frac{1}{\e}\)\geq \frac{\eta}{\e}
\end{align}
and 
\begin{align}
	&\frac{\e\thi}{2} + \frac{1}{4\e} \leq \frac{2}{\lam\e},
\end{align}
for both $i = 1,2$, which gives (\ref{G_nu}) in Lemma \ref{energy est}.

\end{proof}

\subsubsection{The proof of Lemma \ref{energy est of t}}

If one takes $\pt^\nu$ and $\pt^\nu\pt_x$  to (\ref{micro eq with t}), and multiplies $\pt^\nu_\bz h$ and $\pt^\nu\pt_xh$ respectively, then integrates them over $x,v$ and adds the two equations together, one has
\begin{align}
	&\frac{\e}{2}\pt_t \ll \pt^{\nu}h \rl_V^2 +\frac{\lam}{\e} \ll (1-\Pi)\pt^{\nu}h\rl_{V,\o}^2
	\nonumber\\
	\leq& \la \pt^\nu\(E\(\pt_v-\frac{v}{2}\)h \), \pt^\nu h\ra
	+\la \pt^\nu\pt_x\(E\(\pt_v-\frac{v}{2}\)h \), \pt^\nu \pt_xh\ra, \nonumber\\
	&-\underbrace{\la v\sM \pt^\nu\(\(E - \einf\)e^{-\pinf}\), \pt^\nu h\ra}_{IV} 
	-\underbrace{\la v\sM \pt^\nu\pt_x\(\(E - \einf\)e^{-\pinf} \), \pt^\nu \pt_xh\ra}_{V}.
	\label{eq: proc_1_t}
\end{align}

If one multiplies $\sM$ and $v\sM$ to (\ref{micro eq with t}), and integrates it  over $v$, then one gets
\begin{empheq}[left=\empheqlbrace]{align}
&\e\pt_t\s + \pt_xu = 0
	\label{continuity of t}\\
&\e\pt_tu + \pt_x\s + \int v^2\sM (1-\Pi)\pt_xh dv + \frac{1}{\e}u = -E\s -  \(E - \einf\)e^{-\pinf}.
	\label{macro eq of t}
\end{empheq}

Then if one takes $\pt^\nu$ to (\ref{macro eq of t}), and multiplies $\pt^\nu \pt_x\s$, then integrates it over $x,v$, one has
\begin{align}
	&\pt_t \(\e\la \pt^{\nu}u, \pt^{\nu}\pt_x\s \ra + \frac{1}{2}\ll \pt^{\nu}\s\rl^2\)  + \frac{1}{2}\ll \pt^{\nu}\pt_x\s\rl^2\nonumber\\
	 \leq& \ll  \pt^{\nu}\pt_xu\rl^2 + \frac{1}{2}\ll (1-\Pi)\pt^{\nu}h \rl_{V,\o}^2 - \la \pt^{\nu}(E\s), \pt^{\nu}\pt_x\s\ra 
	 - \underbrace{\la \pt^\nu\( \(E - \einf\)e^{-\pinf}\), \pt^\nu\pt_x\s \ra}_{VI}.
	 \label{eq: proc_2_t}
\end{align}

By comparing (\ref{eq: proc_1_t}) and (\ref{eq: proc_2_t}) with (\ref{eq: proc_1}) and (\ref{eq: proc_2}), we actually only need to bound terms $IV, V, VI$. Similar to (\ref{est}), one has the following estimates for $G_i^\nu$, 
\begin{align}
	&\pt_t G_i^{\nu} +\frac{\lam\theta_i}{\e} \ll (1-\Pi)\pt^{\nu}h\rl_{V,\o}^2 + \frac{ C_s}{4\e}\ll \pt^{\nu}\s\rl_V^2
	\nonumber\\
	\leq&  \frac{3}{4\e}\ll (1-\Pi)\pt^{\nu}h \rl_{V,\o}^2 + \l[\theta_i\(I+II\)-\frac{1}{2\e}III \r]+ \l[\theta_i\(IV+V\)-\frac{1}{2\e}VI\r]. 
	\label{est of t}
\end{align}

First notice, for $i = 0, 1$,
\begin{align}
	&\la v\sM \pt^\nu\pt_x^i\(\(E - \einf\)e^{-\pinf}\), \pt^\nu\pt_x^i\s\sM\ra \nonumber\\
	=&   \(\int vM\,dv\) \la  \pt^\nu\pt_x^i\(\(E - \einf\)e^{-\pinf}\), \pt^\nu\pt_x^i\s \ra = 0.
\end{align}
So one can break 
\begin{equation*}
    \pt^\nu\pt_x^i h = (1-\Pi)\pt^\nu\pt_x^i h +  \pt^\nu\pt_x^i\s\sM
\end{equation*}
in $(IV+ V)$, therefore
\begin{align}
\label{est_45}
	IV+V
	=&\sum_{i = 0,1} \la v\sM \pt^\nu\pt_x^i\(\(E - \einf\)e^{-\pinf}\), (1-\Pi)\pt^\nu\pt_x^i h\ra\nonumber\\
	\leq &\frac{\e}{2C_E}\ll \pt^\nu\(\(E - \einf\)e^{-\pinf}\) \rl^2_V + \frac{C_E}{2\e}\ll (1-\Pi)\pt^\nu h\rl_V^2,
\end{align}
For the term $VI$, one can bound it by 
\begin{align}
\label{est_6}
	VI\leq &\frac{1}{C_E}\ll \pt^\nu\( \(E - \einf\)e^{-\pinf}\)\rl^2 + \frac{C_E}{4}\ll \pt^\nu\pt_x\s \rl^2.
\end{align}
For $|\nu| = 0$, plug (\ref{est_45}) and (\ref{est_6}) into (\ref{est of t}), and based on the estimates (\ref{appendix_B_1}) we have already got in Appendices $\ref{proof of energy est of t}$, one has
\begin{align}
	&\pt_t G_i^0 +\frac{\lam\theta_i}{\e} \ll (1-\Pi)h\rl_{V,\o}^2 + \frac{ C_s}{4\e}\ll \s\rl_V^2
	\nonumber\\
	\leq&  \frac{3}{4\e}\ll (1-\Pi)h \rl_{V,\o}^2 + C_E\(\(\frac{\e\thi}{2} + \frac{3}{8\e}\)\ll \s \rl_V^2 + \frac{2\thi}{\e}\ll (1-\Pi) h\rl_{V,\o}^2\)\nonumber\\
	&+\frac{\e\thi}{2C_E}\ll\(E - \einf\)e^{-\pinf} \rl^2_V + \frac{1}{2\e C_E}\ll \( E - \einf\)e^{-\pinf}\rl^2,
\end{align}
which implies,
\begin{align}
	&\pt_t G_i^{0} + C^0_h \ll (1-\Pi)h\rl_{V,\o}^2 + C^0_\s\ll \s\rl_V^2
	\leq \frac{1}{\e C_E}\ll\(E - \einf\)e^{-\pinf} \rl^2_V,\label{final_1 of t}
\end{align}
where 
\begin{align}
\label{def of cons}
	&C^0_h = \frac{\lam\theta_i}{\e} -\frac{3}{4\e} - \frac{2C_E\thi}{\e},\qd C^0_\s \leq  \frac{ C_s}{4\e} -  C_E\(\frac{\e\thi}{2} + \frac{5}{8\e}\).
\end{align}
Since $C_E\leq \frac{\lam C_s}{8} \leq \frac{\lam}{16}$, so $\lam - 2C_E  \geq \frac{7}{8}\lam$, which implies $\(\lam - C_E\)\frac{\thi}{\e} \geq 1$ for both $i = 1,2$. Therefore $C_h^0 \geq \frac{1}{4\e}$. Since $\frac{C_s}{4} - \frac{3C_E}{8} \geq \frac{11}{64}C_s$, and $\e C_E\thi \leq \frac{\e C_s}{7}$ for both $i = 1,2$, therefore $C_\s^0 \geq \frac{11C_s}{64\e} - \frac{\e C_s}{14} \geq \frac{C_s}{10\e}$. Hence plug back $C^0_h, C^0_\s \geq \frac{C_s}{10\e}$ to (\ref{final_1 of t}), one has,
\begin{align}
	&\pt_t G_i^{0} + \frac{\eta}{\e}\ll h\rl_{V,\o}^2 \leq  \frac{1}{\e C_E}\ll\(E - \einf\)e^{-\pinf} \rl^2_V,\qd \text{for } \eta = \frac{C_s}{10}.
\end{align}

For $|\nu| > 0$, plug (\ref{est_45}) and (\ref{est_6}) into (\ref{est of t}), and based on the estimates (\ref{proc_24}), one has
\begin{align}
	&\pt_t G^\nu_i + \frac{\lam\theta_i}{\e} \ll (1-\Pi)\pt^\nu h\rl_{V,\o}^2 + \frac{ C_s}{4\e}\ll \pt^\nu\s\rl_V^2\nonumber\\
	\leq & \frac{3}{4\e} \ll (1-\Pi)\pt^\nu h\rl_{V,\o}^2 + C_E\(\( \frac{\e\thi}{2}+\frac{5}{8\e}\)\ll \pt^{\nu}\s \rl_V^2+ \frac{2\thi}{\e}\ll (1-\Pi)\pt^{\nu}h \rl_{V,\o} ^2\) \nonumber\\
	&+ \frac{1}{\e C_E}\ll \pt^\nu\(\(E - \einf\)e^{-\pinf}\) \rl^2_V +\frac{2}{\lam \e}\sum_{\nu_j\neq0} \nu_j^2C_j\ll \pt^{\nu-e_j}h \rl_V^2,
\end{align}
which implies
\begin{align}
	&\pt_t G^\nu_i + C^\nu_h \ll (1-\Pi)\pt^\nu h\rl_{V,\o}^2 + C^\nu_\s\ll \pt^\nu\s\rl_V^2\nonumber\\
	\leq&\frac{2}{\lam \e} \sum_{\nu_j\neq0} \nu_j^2C_j\ll \pt^{\nu-e_j}h \rl_V^2
	+ \frac{1}{\e C_E}\ll \pt^\nu\(\(E - \einf\)e^{-\pinf}\) \rl^2_V ,
\end{align}
for the same $C^\nu_h, C^\nu_h$ defined in (\ref{def of cons}). This completes the proof.

\bibliographystyle{plain}

\end{document}